\documentclass{article}
\usepackage[utf8]{inputenc}
\usepackage{amsmath}
\usepackage{amsthm}
\usepackage{amssymb}
\usepackage{amsfonts}
\usepackage{authblk}
\usepackage{mathtools}
\usepackage{stmaryrd}
\usepackage{tikz-cd}
\usepackage{enumerate}
\usepackage{hyperref} \hypersetup{colorlinks=true,citecolor=[rgb]{0.153,0.255,0.545},urlcolor=black,colorlinks=true,linkcolor=[rgb]{0.153,0.255,0.545}}
\usepackage{bbm}
\usepackage{multirow}
\usepackage{multicol}
\usepackage{pdflscape}
\usepackage{comment}
\usepackage[all,2cell]{xy}
\UseTwocells

\usepackage{todonotes}

\usepackage[top=1in, bottom=1.25in, left=1.25in, right=1.25in]{geometry}

\newtheorem{theorem}{Theorem}[section]
\newtheorem*{theorem*}{Theorem}
\newtheorem{proposition}[theorem]{Proposition}
\newtheorem{lemma}[theorem]{Lemma}
\newtheorem{corollary}[theorem]{Corollary}
\theoremstyle{definition}
\newtheorem{example}[theorem]{Example}
\newtheorem{remark}[theorem]{Remark}

\newtheorem{definition}[theorem]{Definition}

\newcommand{\ev}{\textup{ev}}
\newcommand{\Bd}{\textup{\bf{}Bd}}
\newcommand{\Bdp}{\textup{\bf{}Bdp}}
\newcommand{\Tb}{\textup{\bf{}Tb}}
\newcommand{\id}{\textup{id}}

\newcommand{\PGL}{\textup{PGL}}
\newcommand{\AGL}{\textup{AGL}}
\newcommand{\GL}{\textup{GL}}
\newcommand{\SL}{\textup{SL}}
\newcommand{\BG}{{\textup{B}G}}
\newcommand{\CC}{\mathbb{C}}
\renewcommand{\AA}{\mathbb{A}}
\newcommand{\Stck}{\textup{\bf{}Stck}}
\newcommand{\RStck}{\textup{\bf{}RStck}}
\newcommand{\Spc}{\textup{\bf{}Spc}}
\newcommand{\K}{\textup{K}}
\newcommand{\GG}{\mathbb{G}}
\newcommand{\FF}{\mathbb{F}}
\newcommand{\PP}{\mathbb{P}}
\newcommand{\LL}{q}
\newcommand{\ZZ}{\mathbb{Z}}
\newcommand{\Var}{\textup{\bf{}Var}}

\newcommand{\Grpd}{\textbf{Grpd}}
\newcommand{\Span}{\textup{Span}}
\newcommand{\point}{\star}
\newcommand{\Hom}{\textup{Hom}}
\newcommand{\Mod}{\textup{\bf Mod}}

\newcommand{\cC}{\mathcal{C}}

\newcommand{\cF}{\mathcal{F}}

\newcommand{\B}{\textup{B}}

\DeclareMathOperator{\tr}{tr}
\DeclareMathOperator{\Spec}{Spec}

\DeclareMathOperator{\orb}{orb}
\DeclareMathOperator{\Aut}{Aut}

\newcommand{\smatrix}[1]{\left(\begin{smallmatrix} #1 \end{smallmatrix}\right)}

\newcommand{\bdscale}{0.5}

\newcommand\bdpantsleft[1][\bdscale] {
\begin{tikzpicture}[semithick, scale=\bdscale, baseline=-0.5ex]
\begin{scope}
    \draw (-1,0.5) ellipse (0.2cm and 0.4cm);
    \draw (-1,-0.5) ellipse (0.2cm and 0.4cm);
    \draw (1,0) ellipse (0.2cm and 0.4cm);
    \draw (-1,0.9) .. controls (0,0.9) and (0,0.4) .. (1,0.4);
    \draw (-1,-0.9) .. controls (0,-0.9) and (0,-0.4) .. (1,-0.4);
    \draw (-1,0.1) .. controls (-0.2,0.1) and (-0.2,-0.1) .. (-1,-0.1);
\end{scope}
\end{tikzpicture}
}

\newcommand\bdpantsright[1][\bdscale] {
\begin{tikzpicture}[semithick, scale=#1, baseline=-0.5ex]
\begin{scope}
    \draw (-1,0) ellipse (0.2cm and 0.4cm);
    \draw (1,0.5) ellipse (0.2cm and 0.4cm);
    \draw (1,-0.5) ellipse (0.2cm and 0.4cm);
    \draw (-1,0.4) .. controls (0,0.4) and (0,0.9) .. (1,0.9);
    \draw (-1,-0.4) .. controls (0,-0.4) and (0,-0.9) .. (1,-0.9);
    \draw (1,0.1) .. controls (0.2,0.1) and (0.2,-0.1) .. (1,-0.1);
\end{scope}
\end{tikzpicture}
}

\def\bdgenerictube#1 {
\begin{tikzpicture}[semithick, scale=\bdscale, baseline=-0.5ex]
\begin{scope}
    \draw (-1,0) ellipse (0.2cm and 0.4cm);
    \draw (-1,0.4) cos (-0.875, 0.5) sin (-0.75,0.6) cos (-0.625,0.5) sin (-0.5,0.4) cos (-0.375,0.5) sin (-0.25, 0.6) cos (-0.125,0.5) sin (0, 0.4) cos (0.125,0.5) sin (0.25,0.6) cos (0.375,0.5) sin (0.5,0.4) cos (0.625,0.5) sin (0.75,0.6) cos (0.875,0.5) sin (1,0.4);
    \draw (-1,-0.4) cos (-0.875, -0.5) sin (-0.75,-0.6) cos (-0.625,-0.5) sin (-0.5,-0.4) cos (-0.375,-0.5) sin (-0.25, -0.6) cos (-0.125,-0.5) sin (0,-0.4) cos (0.125,-0.5) sin (0.25,-0.6) cos (0.375,-0.5) sin (0.5,-0.4) cos (0.625,-0.5) sin (0.75,-0.6) cos (0.875,-0.5) sin (1,-0.4);
    \draw (1,0) ellipse (0.2cm and 0.4cm);
    \draw (0,0) node {$#1$};
\end{scope}
\end{tikzpicture}
}

\newcommand\bdpgenus[1][\bdscale] {
\begin{tikzpicture}[semithick, scale=2.0*(#1), baseline=-0.5ex]
\begin{scope}
    \draw (-1,0) ellipse (0.2cm and 0.4cm);
    \draw (-1,0.4) .. controls (-0.5,0.45) and (0.5,0.45) .. (1,0.4);
    \draw (-1,-0.4) .. controls (-0.5,-0.45) and (0.5,-0.45) .. (1,-0.4);
    \draw (-0.5,0.1) .. controls (-0.5,-0.125) and (0.5,-0.125) .. (0.5,0.1);
    \draw (-0.4,0.0) .. controls (-0.4,0.0625) and (0.4,0.0625) .. (0.4,0.0);
    \draw (1,0.4) arc (90:-90:0.2cm and 0.4cm);
    \draw[dashed] (1,0.4) arc (90:270:0.2cm and 0.4cm);
    \draw[black,fill=black] (-0.8,0) circle (.4ex);
    \draw[black,fill=black] (1.2,0) circle (.4ex);
\end{scope}
\end{tikzpicture}
}

\newcommand\bdpcupleft[1][\bdscale] {
\begin{tikzpicture}[semithick, scale=2.0*(#1), baseline=-0.5ex]
\begin{scope}
    \draw (0,0) ellipse (0.2cm and 0.4cm);
    \draw (0,-0.4) arc (-90:90:0.75cm and 0.4cm);
    \draw[black,fill=black] (0.2,0) circle (.4ex);
\end{scope}
\end{tikzpicture}
}

\newcommand\bdpcupright[1][\bdscale] {
\begin{tikzpicture}[semithick, scale=2.0*(#1), baseline=-0.5ex]
\begin{scope}
    \draw (0,0.4) arc (90:-90:0.2cm and 0.4cm);
    \draw (0,0.4) arc (90:270:0.75cm and 0.4cm);
    \draw[dashed] (0,0.4) arc (90:270:0.2cm and 0.4cm);
    \draw[black,fill=black] (0.2,0) circle (.4ex);
\end{scope}
\end{tikzpicture}
}

\newcommand\bdppantsleft[1][\bdscale] {
\begin{tikzpicture}[semithick, scale=#1, baseline=-0.5ex]
\begin{scope}
    \draw (-1,0.5) ellipse (0.2cm and 0.4cm);
    \draw (-1,-0.5) ellipse (0.2cm and 0.4cm);
    \draw (1,0) ellipse (0.2cm and 0.4cm);
    \draw (-1,0.9) .. controls (0,0.9) and (0,0.4) .. (1,0.4);
    \draw (-1,-0.9) .. controls (0,-0.9) and (0,-0.4) .. (1,-0.4);
    \draw (-1,0.1) .. controls (-0.2,0.1) and (-0.2,-0.1) .. (-1,-0.1);
    \draw[black,fill=black] (-0.8,0.5) circle (.4ex);
    \draw[black,fill=black] (-0.8,-0.5) circle (.4ex);
    \draw[black,fill=black] (0.8,0) circle (.4ex);
\end{scope}
\end{tikzpicture}
}

\newcommand\bdppantsright[1][\bdscale] {
\begin{tikzpicture}[semithick, scale=#1, baseline=-0.5ex]
\begin{scope}
    \draw (-1,0) ellipse (0.2cm and 0.4cm);
    \draw (1,0.5) ellipse (0.2cm and 0.4cm);
    \draw (1,-0.5) ellipse (0.2cm and 0.4cm);
    \draw (-1,0.4) .. controls (0,0.4) and (0,0.9) .. (1,0.9);
    \draw (-1,-0.4) .. controls (0,-0.4) and (0,-0.9) .. (1,-0.9);
    \draw (1,0.1) .. controls (0.2,0.1) and (0.2,-0.1) .. (1,-0.1);
    \draw[black,fill=black] (-0.8,0) circle (.4ex);
    \draw[black,fill=black] (1.2,0.5) circle (.4ex);
    \draw[black,fill=black] (1.2,-0.5) circle (.4ex);
\end{scope}
\end{tikzpicture}
}

\newcommand{\bdpgenuscappedleft}[1][\bdscale] {
\begin{tikzpicture}[semithick, scale=#1, baseline=-0.5ex]
\begin{scope}
    \draw (-0.5,0.1) .. controls (-0.5,-0.125) and (0.5,-0.125) .. (0.5,0.1);
    \draw (-0.4,0.0) .. controls (-0.4,0.0625) and (0.4,0.0625) .. (0.4,0.0);
    \draw (-1,0) ellipse (0.2cm and 0.4cm);
    \draw (-1,0.4) -- (0.3, 0.4);
    \draw (0.3,0.4) arc (270:90:-0.75cm and -0.4cm);
    \draw (0.3, -0.4) -- (-1,-0.4);
    \draw[black,fill=black] (-0.8,0) circle (.4ex);
\end{scope}
\end{tikzpicture}
}

\newcommand{\bdppgenuscappedleft}[1][\bdscale] {
\begin{tikzpicture}[semithick, scale=#1, baseline=-0.5ex]
\begin{scope}
    \draw (-0.5,0.1) .. controls (-0.5,-0.125) and (0.5,-0.125) .. (0.5,0.1);
    \draw (-0.4,0.0) .. controls (-0.4,0.0625) and (0.4,0.0625) .. (0.4,0.0);
    \draw (-1,0) ellipse (0.2cm and 0.4cm);
    \draw (-1,0.4) -- (0.3, 0.4);
    \draw (0.3,0.4) arc (270:90:-0.75cm and -0.4cm);
    \draw (0.3, -0.4) -- (-1,-0.4);
    \draw[black,fill=black] (-0.8,0) circle (.4ex);
    \draw[black,fill=black] (0.8,0) circle (.4ex);
\end{scope}
\end{tikzpicture}
}

\newcommand\bdppgenus[1][\bdscale] {
\begin{tikzpicture}[semithick, scale=2.0*(#1), baseline=-0.5ex]
\begin{scope}
    \draw (-1,0) ellipse (0.2cm and 0.4cm);
    \draw (-1,0.4) .. controls (-0.5,0.45) and (0.5,0.45) .. (1,0.4);
    \draw (-1,-0.4) .. controls (-0.5,-0.45) and (0.5,-0.45) .. (1,-0.4);
    \draw (-0.5,0.1) .. controls (-0.5,-0.125) and (0.5,-0.125) .. (0.5,0.1);
    \draw (-0.4,0.0) .. controls (-0.4,0.0625) and (0.4,0.0625) .. (0.4,0.0);
    \draw (1,0.4) arc (90:-90:0.2cm and 0.4cm);
    \draw[dashed] (1,0.4) arc (90:270:0.2cm and 0.4cm);
    \draw[black,fill=black] (-0.8,0) circle (.4ex);
    \draw[black,fill=black] (1.2,0) circle (.4ex);
    \draw[black,fill=black] (0.5,-0.25) circle (.4ex);
    \draw[black,fill=black] (-0.5,-0.25) circle (.4ex);
\end{scope}
\end{tikzpicture}
}

\def\bdpcylinder {
\begin{tikzpicture}[semithick, scale=\bdscale, baseline=-0.5ex]
\begin{scope}
    \draw (-1,0) ellipse (0.2cm and 0.4cm);
    \draw (-1,0.4) -- (1,0.4);
    \draw (-1,-0.4) -- (1,-0.4);
    \draw (1,0.4) arc (90:-90:0.2cm and 0.4cm);
    \draw[dashed] (1,0.4) arc (90:270:0.2cm and 0.4cm);
    \draw[black,fill=black] (-0.8,0) circle (.4ex);
    \draw[black,fill=black] (1.2,0) circle (.4ex);
\end{scope}
\end{tikzpicture}
}

\newcommand{\bdppantsleftsmall}{\bdppantsleft[0.25]}
\newcommand{\bdppantsrightsmall}{\bdppantsright[0.25]}
\newcommand{\bdpantsrightsmall}{\bdpantsright[0.25]}
\newcommand{\bdpgenussmall}{\bdpgenus[0.22]}
\newcommand{\bdpcupleftsmall}{\bdpcupleft[0.22]}
\newcommand{\bdpcuprightsmall}{\bdpcupright[0.22]}

\title{\bf{}Virtual Classes of Character Stacks}
\author[$\dagger$]{\'Angel Gonz\'alez-Prieto}
\author[$\ddagger$]{Márton Hablicsek}
\author[$\star$]{Jesse Vogel}
\affil[$\dagger$]{\footnotesize Department of Algebra, Geometry and Topology, Universidad Complutense de Madrid, and Instituto de Ciencias Matem\'aticas (CSIC-UAM-UCM-UC3M), Plaza de Ciencias 3, Ciudad Universitaria. 28040 Madrid, Spain, angelgonzalezprieto@ucm.es}
\affil[$\ddagger$]{\footnotesize Department of Mathematics, Leiden University, Niels Bohrweg 1, 2333 CA Leiden, Netherlands, hablicsekhm@math.leidenuniv.nl, Corresponding author}
\affil[$\star$]{\footnotesize Department of Mathematics, Leiden University, Niels Bohrweg 1, 2333 CA Leiden, Netherlands, j.t.vogel@math.leidenuniv.nl}
\date{\today}

\begin{document}

\maketitle

\begin{abstract}
    In this paper, we extend the Topological Quantum Field Theory developed by González-Prieto, Logares, and Muñoz for computing virtual classes of $G$-representation varieties of closed orientable surfaces in the Grothendieck ring of varieties to the setting of the character stacks. To this aim, we define a suitable Grothendieck ring of representable stacks, over which this Topological Quantum Field Theory is defined. In this way, we compute the virtual class of the character stack over $\BG$, that is, a motivic decomposition of the representation variety with respect to the natural adjoint action.
    
    We apply this framework in two cases providing explicit expressions for the virtual classes of the character stacks of closed orientable surfaces of arbitrary genus. First, in the case of the affine linear group of rank $1$, the virtual class of the character stack fully remembers the natural adjoint action, in particular, the virtual class of the character variety can be straightforwardly derived. Second, we consider the non-connected group $\GG_m \rtimes \ZZ/2\ZZ$, and we show how our theory allows us to compute motivic information of the character stacks where the classical na\"ive point-counting method fails.
    %
\end{abstract}
 \section{Introduction}
Let $G$ be an algebraic group over a field $k$, and $M$ a smooth manifold with finitely generated  fundamental group (e.g.\ $M$ is compact). The collection of representations of the fundamental group $\pi_1(M)$ into $G$ forms an algebraic variety,
\[ R_G(M) = \Hom(\pi_1(M), G) , \]
called the $G$-representation variety of $M$. The geometry of this representation variety has been widely studied in the last years. For instance, when $M = S^3-K$ is the complement of a knot, then $R_G(M)$ provides precious knot theoretic information that can be used to generate knot invariants \cite{cooper1998representation,przytycki1997skein}. More generally, when $M$ is a $3$-fold, then the discrete and faithful representations of the $\PGL_2(\CC)$-representation variety are precisely the ways of endowing $M$ with a hyperbolic structure, a fact that has been exploited to prove deep results in hyperbolic geometry \cite{culler1983varieties}. 

Nonetheless, the representation variety $R_G(M)$ only parametrizes raw representations, so isomorphic representations appear in $R_G(M)$ as different points. To remove this redundancy, we must consider the adjoint action of $G$ on $R_G(M)$, given by conjugation, whose orbits are precisely the representations up to isomorphism. However, in general, the quotient of an algebraic variety under the action of an algebraic group is not an algebraic variety, so the orbit space $R_G(M)/G$ is no longer a variety. To overcome this difficulty, two different approaches can be followed: through Geometric Invariant Theory (GIT) or using quotient stacks.

In the former approach, one takes the quotient as the spectrum of the ring of $G$-invariant functions on $R_G(M)$. If $G$ is a reductive group, this spectrum defines an algebraic variety acting as a sort of weak quotient, the so-called GIT quotient \cite{mumford1994geometric,newstead1978introduction}, usually denoted by
\[ \chi_G(M) = R_G(M)\sslash G , \]
and known as the character variety or the Betti moduli space. The case in which $G = \GL_r(\CC)$ and $M = \Sigma_g$ is the genus $g$ closed orientable surface plays a central role in the non-abelian Hodge correspondence: the character variety $\chi_G(\Sigma_g)$ turns out to be isomorphic to the moduli space of flat connections on $\Sigma_g$ \cite{simpson1994moduli,simpson1994moduliII} and homeomorphic to the moduli space of Higgs bundles on $\Sigma_g$ (for a given complex structure) \cite{corlette1988flat}. For these reasons, the algebraic structure of the character variety $\chi_G(\Sigma_g)$ is objective of intense research.

On the other hand, the latter approach does not seek a variety that behaves as a quotient, but rather it enlarges the category of varieties to the category of stacks \cite{neumann2009algebraic}, in which orbit spaces lie naturally. The solution is then to look at $R_G(M)/G$ as the moduli problem of parametrizing principal $G$-bundles $P$ with an equivariant map $P \to R_G(M)$ onto our `model space'. This gives rise to an algebraic stack, the quotient stack
\[ \mathfrak{X}_G(M) = [R_G(M) / G], \]
also known as the character stack, whose geometry can be understood through the `algebraic chart' $R_G(M) \to \mathfrak{X}_G(M)$.

Despite the importance of these character stacks, very little is known about their geometry. To the best of our knowledge, the only known information is that, for surfaces, their point count over finite fields is so-called Polynomial On Residue Classes (PORC), namely, there exists a finite family of polynomials counting the $\FF_q$-points $\mathfrak{X}_G(M)(\FF_q)$. Inspired by the Weil conjectures, this allows to compute the $E$-polynomials of the character stacks over the complex numbers \cite{bridger2022character}. However, virtually nothing is known about more complicated invariants of $\mathfrak{X}_G(M)$. In particular, no explicit calculation has been done so far, even in the simplest cases, and importantly the adjoint $G$-action cannot be tracked with the existing methods.


\subsection*{Motivic theory of quotient stacks}

The most general motivic invariants that one would like to compute for character stacks are their virtual classes. Roughly speaking, one can form the so-called Grothendieck ring of stacks, denoted by $\K(\Stck_k)$. This ring is generated by isomorphism classes $[\mathfrak{X}]$ of stacks $\mathfrak{X}$ up to cut-and-paste relations. The image $[\mathfrak{X}] \in \K(\Stck_k)$ of a stack is usually referred to as the virtual class or the motive of $\mathfrak{X}$. It encodes all the possible motivic invariants of $\mathfrak{X}$, in the sense that if $\chi: \Stck_k \to R$ is any isomorphism invariant of stacks taking values in a ring $R$ and satisfying $\chi(\mathfrak{X}) = \chi(\mathfrak{Z}) + \chi(\mathfrak{X} \setminus \mathfrak{Z})$ and $\chi(\mathfrak{X} \times \mathfrak{Y}) = \chi(\mathfrak{X}) \cdot \chi(\mathfrak{Y})$ for all stacks $\mathfrak{X}$ and $\mathfrak{Y}$ and closed substacks $\mathfrak{Z} \subset \mathfrak{X}$, then there exists a unique ring homomorphism $\overline{\chi}: \K(\Stck_k) \to R$ such that the following diagram commutes
\[ \begin{tikzcd}
    \Stck_k\arrow{r}{\chi} \arrow[d, hook]& R \\
    \K(\Stck_k)  \arrow[ru, "\overline{\chi}"']
\end{tikzcd} \]

However, even in this framework, we completely lose the information of the $G$-action on $X$ in the case of a quotient stack $\mathfrak{X} = [X/G]$.
In order to keep track of the action of $G$ on $X$, notice that this action is classified by the natural morphism $\mathfrak{X} \to \BG$ into the classifying stack $\BG = [\star/G]$. Indeed, roughly speaking, a morphism $\mathfrak{X} \to \BG$ is the same as an algebraic space $X$ equipped with a $G$-action, in such a way that $\mathfrak{X} = [X/G]$ (for a precise statement, see Lemma \ref{lem:quotient-G-torsor}). Hence, if we want to remember the action, we should pass to the relative setting and study $\mathfrak{X}$ not as an absolute stack, but as a $\BG$-stack and, thus, the natural virtual class to study is $[\mathfrak{X} \to \BG] \in \K(\Stck/\BG)$ in the Grothendieck ring of $\BG$-stacks. In some sense, $\K(\Stck/\BG)$ must be seen as the $G$-equivariant version of the absolute ring $\K(\Stck_k)$, where now only $G$-invariant decompositions are allowed.

Coming back to our representation theoretic setting, in order to understand the motivic theory of character stacks, in Section \ref{sec:grothendieck-ring}, we define the \textit{Grothendieck ring of representable stacks} over the classifying stack $\BG$, denoted $\K(\RStck/\BG)$, generated by separated algebraic $G$-spaces and we consider 
the virtual classes
$$
    [\mathfrak{X}_G(M)] \in \K(\RStck/\BG).
$$
A precise understanding of this virtual class provides a lot of important equivariant information, such as how to decompose the representation variety $R_G(M)$ into $G$-equivariant pieces, how to stratify it according to the stabilizers of the action and how $G$ acts on each of these pieces. From these data, one can try to understand subtle properties such as the locus of irreducible representations or to envisage the GIT quotient.

The aim of this work is precisely to provide a general method able to perform these virtual class calculations in an effective way.
To be precise, in this paper we construct a Topological Quantum Field Theory (TQFT) computing such virtual classes, as proven in Theorem \ref{thm:tqft-stacky}.

\begin{theorem*}
    Let $n \geq 1$ and $G$ an algebraic group. There exists a lax monoidal Topological Quantum Field Theory
    \[ Z \colon \Bdp_n \to \K(\RStck/\BG)\textup{-}\Mod, \]
    computing the virtual classes of $G$-character stacks of closed orientable $n$-dimensional manifolds.
\end{theorem*}

Here, $\Bdp_n$ is the category of $n$-dimensional \textit{pointed} orientable bordisms and $\K(\RStck/\BG)\textup{-}\Mod$ denotes the category of modules over the ring $\K(\RStck/\BG)$. The way in which $Z$ computes virtual classes is the following. Suppose that $M$ is a closed $n$-dimensional manifold, which can be seen as a bordism, that is, morphism of $\Bdp_n$, $M : \varnothing \to \varnothing$. Under the TQFT, this gives rise to a $\K(\RStck/\BG)$-linear map $Z(M): \K(\RStck/\BG)\to \K(\RStck/\BG)$, so it is given by multiplication by some fixed element of $\K(\RStck/\BG)$: such a factor is precisely the desired virtual class $[\mathfrak{X}_G(W)] \in \K(\RStck/\BG)$.

This TQFT generalizes previous constructions known in the literature to the stacky framework. In \cite{gonlogmun20}, a TQFT computing $E$-polynomials of complex representation varieties was built, in \cite{arXiv181009714} and \cite{gon20} such TQFT was adapted to work also in the parabolic setting, in \cite{gonzalez2020character} to surfaces with conic singularities, in \cite{vogel2020representation} for non-orientable surfaces and in \cite{habvog20} for $G$ the group of upper triangular matrices of rank $\leq 4$. Notice that no TQFT can be constructed to compute virtual classes of character varieties, since the GIT quotient identification of orbits prevents them from preserving pullbacks.


The main difficulty we face in this paper is that neither the virtual class of representation varieties nor of character varieties is a natural output of the quantum method. In fact, in the recent paper \cite{gonzalez2023arithmetic}, the authors of this article showed that this TQFT-based method naturally extends to compute virtual classes of character stacks in the absolute Grothendieck ring $\K(\Stck_k)$. 
However, this extension loses the information of how $G$ acts via the adjoint action on the representation variety.
Thus, for instance, the character variety (the GIT quotient) cannot be studied or understood using this framework. To resolve this problem, an ad-hoc method needs to be introduced, namely, adding basepoints to bordisms and new ``cone-like'' bordisms to the framework. However, since the goal of the paper is precisely to understand the $G$-equivariant theory on the representation variety, we decided to use a different framework than \cite{gonzalez2023arithmetic}. The framework presented in this paper allows us to perform explicit computations of virtual classes of character stacks as $\BG$-stacks or, in other words, representation varieties equipped with the $G$-action. As a drawback, the framework is somehow artificial, since all the 2-categorical data naturally presented in the character stack are ignored.

Be that as it may, the virtual class of the character stack $\mathfrak{X}_G(M)$ in $\K(\RStck/\BG)$, computed through the TQFT developed in this work, possesses a lot of information that cannot be obtained from the class of the representation variety $R_G(M)$ in the Grothendieck ring $\K(\Var_k)$ of algebraic varieties. For instance, for any subgroup $H \subset G$, there is the $\K(\RStck_k)$-module morphism
\[ (-)^H \colon \K(\RStck/\BG) \to \K(\RStck_k) , \]
which sends a quotient stack $[X/G]$ to the invariant locus $X^H \subset X$ under the subgroup $H$. In the case of $H = G$, one recovers the fixed points of the representation variety under the group action of $G$, and in the case of $H = \{ 1 \}$, the invariant locus of $[X/G]$ is $[X]$ itself, which allows us to recover the class of the representation variety $[R_G(M)]$ from $[\mathfrak{X}_G(M)]$.
More generally, in this way one can recover the classes of the loci having certain stabilizer, and in fact, the virtual class of the character stacks over $\BG$ remembers the natural adjoint action of $G$ on the representation variety providing a motivic decomposition with respect to this action.


Another piece of information that can be obtained from the class of the character stack is its image under the evaluation map defined in Section \ref{sec:equivariant-setting},
\[ \textup{ev} \colon \K(\RStck/\BG) \to \hat{\K}(\Var_k) , \]
where $\hat{\K}(\Var_k)$ denotes the localization of the Grothendieck ring of varieties, $\K(\Var_{k})$, by inverting the class of the affine line $q = [\AA_k^1]$ and the classes of the form $q^n-1$. When $G$ is a special group, the evaluation map sends the class of $[X/G]$ to the class $[X]/[G]$ in $\hat{\K}(\Var_k)$, so we have a commutative diagram
\[ \begin{tikzcd}
    \K(\RStck/\BG) \arrow{rr}{\textup{ev}} \arrow[swap]{dr}{(-)^{\{ 1 \}}} & & \hat{\K}(\Var_k) \\
    & \K(\RStck_k) \arrow[swap]{ur}{\cdot [G]^{-1}} &
\end{tikzcd} \]
However, when $G$ is not a special group, the above diagram is not a commutative diagram. In fact, in Section \ref{sec:GGmZZ2} we show the character stacks corresponding to the semi-direct product $\GG_m\rtimes \ZZ/2\ZZ$ provide examples of the failure of the diagram above.


\subsection*{$\AGL_1(k)$-character stacks}
Notice that the effective method of computation of virtual classes derived from the TQFT $Z \colon \Bdp_n \to \K(\RStck/\BG)\textup{-}\Mod$ works for any algebraic group $G$ and any dimension $n$. Hence, to examplify our method, we consider in Section \ref{sec:AGL1-character-stack} the algebraic group $G = \AGL_1(k)$ of affine transformations of the affine line over a field $k$. Explicitly computing the TQFT for surfaces, we find that:
\begin{theorem*}
    The class of the character stack $[\mathfrak{X}_G(\Sigma_g)] \in \K(\RStck/\BG)$ for $G = \AGL_1(k)$ is given by
    \[ [\mathfrak{X}_G(\Sigma_g)] = [\BG] + ((q - 1)^{2g} - 1)[\GG_a/G] + \frac{q^{2 g} - 1}{q - 1}[\GG_m/G] + \frac{\left(q^{2 g - 2} - 1\right) \left(\left(q - 1\right)^{2 g} - 1\right)}{q - 1} [\textup{AGL}_1/G] , \]
    where $q = [\AA^1_k]$ denotes the class of the affine line, with trivial $G$-action.
\end{theorem*}
As immediate applications, we recover the virtual classes of the representation variety for $G = \AGL_1(k)$, and by identifying the GIT quotient as the invariant part of the diagonal subgroup $D \subset \AGL_1(k)$, we compute the class of the character variety, i.e. the GIT quotient $[R_G(\Sigma_g) \sslash G]$, as
\[ [R_G(\Sigma_g) \sslash G] = [\mathfrak{X}_G(\Sigma_G)]^{D} = (q - 1)^{2g} , \]
agreeing with the results of \cite{gonlogmun20} and \cite{habvog20}.

\subsection*{Arithmetic of character stacks of non-connected groups}
An interesting feature appears when one studies character varieties and stacks from an arithmetic lens. A celebrated result of Katz \cite{hausel2008mixed}, presented in an appendix of a paper by Hausel and Rodr\'iguez-Villegas, shows that if the number of points of the character variety on the finite field $\mathbb{F}_q$ of $q$ elements is a polynomial in $q$, then this polynomial is the $E$-polynomial of the complex character variety in the variable $q = uv$. For this reason, multiple works have focused on counting these solutions over finite field with arithmetic arguments, such as \cite{hausel2008mixed} for  $G = \GL_r(\CC)$, \cite{mereb2015polynomials} for $G = \SL_r(\CC)$ or \cite{letellier2020series} for non-orientable surfaces, among others.

This observation has a counterpart for character stacks of connected linear algebraic groups $G$. Lang's theorem  implies that any principal $G$-bundle over $\FF_q$ is trivial. Therefore, if $G$ acts on a separated scheme $X$ of finite type over a finite field $\FF_q$, the number of $\FF_q$-points of the quotient stack $[X/G]$ is simply the quotient of the number of $\FF_q$-points of the schemes $X$ and $G$ \cite{behrend1993lefschetz}, 
\begin{equation}
    \label{eq:pointcstacks}
    \#[X/G](\FF_q) = \frac{\#X(\FF_q)}{\#G(\FF_q)} .
\end{equation}
In this way, using the arithmetic method of Hausel and Rodr\'iguez-Villegas, the above formula can be used to compute the $E$-polynomial of character stacks of connected linear algebraic groups (e.g. \cite{bridger2022character}). Hence, on the level of $E$-polynomials, the character stack does not carry more information than the representation variety. 

However, for non-connected linear algebraic groups, the above point-counting formula fails, already for the simple case of $G = \ZZ/2\ZZ$. Namely, in this case there are exactly two non-isomorphic principal $G$-bundles over $\FF_q$, the trivial bundle and the bundle corresponding to the field extension $\FF_q\to \FF_{q^2}$, so the number of $\FF_q$-points of $\BG$ is 
\[ \#\BG(\FF_q) = \sum_{x \in [\BG(\FF_q)]} \frac{1}{|\Aut_G(x)|} = \frac{1}{2} + \frac{1}{2} = 1 , \] 
while the `naive' point-counting yields
\[ \frac{\#\Spec(\FF_q)(\FF_q)}{\#G(\FF_q)} = \frac{1}{2} . \]
This shows that one needs to be careful in using the arithmetic method for quotient stacks of non-connected groups. In fact, the class $[\B(\ZZ/2\ZZ)]$ in the Grothendieck ring of stacks is 1 (\cite{eke09}), as pointed out by the stacky counting, rather than $\frac{1}{2}$ as predicted by the `naive' counting. 

In Section \ref{sec:GGmZZ2}, we illustrate the above phenomenon explicitly using the linear algebraic group $G = \GG_m \rtimes \ZZ/2\ZZ$, where $\ZZ/2\ZZ$ acts on $\GG_m$ by $x \mapsto x^{-1}$.
In particular, after applying the evaluation map, we find that the class of the character stack is given by  
\[ [\mathfrak{X}_G(\Sigma_g)] = \frac{\left(q - 1\right)^{2 g - 2} \left(2^{2 g + 1} + q - 3\right)}{2} + \frac{\left(q + 1\right)^{2 g - 2} \left(2^{2 g + 1} + q - 1\right)}{2} \in \hat{\K}(\Var_k) , \]
which is different from the quotient $[R_G(\Sigma_g)]/[G]$, as $[R_G(\Sigma)] = (q - 1)^{2g - 1} (2^{2g + 1}+q-3)$.

\subsection*{Acknowledgments}
The second and third authors would like to thank Bas Edixhoven, David Holmes, and Masoud Kamgarpour for important discussions related to the paper; and Matthieu Romagny for the discussion about Lemma \ref{le:romagny}. The first author wants to thank the hospitality of the Department of Mathematics at Universidad Aut\'onoma de Madrid where this work was partially completed. The first author has been partially supported by the Spanish \textit{Ministerio de Ciencia e Innovación} through the research project PID2019-106493RB-I00 (DL-CEMG), by the COMPLEXFUILDS grant awarded by the BBVA Foundation, and by \textit{Comunidad de Madrid} R+D research project PR27/21-029 for young researchers. 
\section{Grothendieck ring of representable stacks}
\label{sec:grothendieck-ring}

In this paper, we shall work on categories of relative stacks with separated and representable morphisms
that will be key for our purposes. For further information about stacks, including the definitions in the absolute setting, please refer to \cite{neumann2009algebraic} or \cite{stacks-project}, among others.

Let $\mathfrak{S}$ be an algebraic (Artin) stack of finite type over a field $k$.
The \textit{category of stacks over $\mathfrak{S}$}, denoted by $\Stck/\mathfrak{S}$, is the following $2$-category.
\begin{itemize}
    \item The objects are pairs $(\mathfrak{X}, \pi)$, where $\mathfrak{X}$ is an algebraic stack of finite type over $k$, and $\pi \colon \mathfrak{X} \to \mathfrak{S}$ is a $1$-morphism of stacks. If the $1$-morphism $\pi$ is understood from the context, we denote the object simply by $\mathfrak{X}$.
    \item A $1$-morphism $(f, \alpha) \colon (\mathfrak{X}, \pi) \to (\mathfrak{X'}, \pi')$ consists of a $1$-morphism of stacks $f \colon \mathfrak{X} \to \mathfrak{X'}$ and a $2$-morphism of stacks $\alpha \colon \pi \Rightarrow \pi' \circ f$.
    \item A $2$-morphism $\mu \colon (f, \alpha) \Rightarrow (g, \beta)$ in $\Stck/\mathfrak{S}$ is a natural isomorphism such that $\pi'(\mu) \circ \alpha = \beta$.
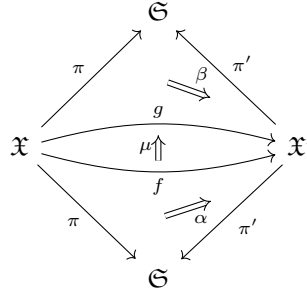
\[ \begin{tikzcd}[column sep=large, row sep = huge]
    & \mathfrak{S} & \\ 
    \mathfrak{X} \arrow[dr, swap, "\pi"{name=U}] \arrow[ur, "\pi"{name=W}] \arrow[rr, swap, bend right=15, "f"{name=A}] \arrow[rr, bend left=15, "g"{name=B}]  \arrow[Rightarrow, shorten=10mm, from=W, to=rr, shift left=1.5ex, "\beta"] \arrow[Rightarrow, shorten = 1.5mm, from=A, to=B, "\mu"] & & \mathfrak{X}' \arrow[dl, "\pi'"]\arrow[ul, swap, "\pi'"] \\
    & \mathfrak{S} \arrow[Rightarrow, swap, shorten=10mm, from=U, to=ru, shift right=2ex, "\alpha"] &
\end{tikzcd} \]
\end{itemize}
Throughout the paper, we work with a special subcategory of $\Stck/\mathfrak{S}$ consisting of representable and separated morphisms of stacks $\pi \colon \mathfrak{X}\to \mathfrak{S}$. 

\begin{definition}
    \label{def:representable_morphism}
    A morphism of stacks $\mathfrak{X}\to \mathfrak{S}$ is \textit{representable} if, for any scheme $T$ over $\mathfrak{S}$, the fiber product $T\times_{\mathfrak{S}}\mathfrak{X}$ is an algebraic space.
    A morphism of stacks $\mathfrak{X}\to \mathfrak{S}$ is \textit{separated} if the diagonal map $\mathfrak{X} \to \mathfrak{X} \times_{\mathfrak{S}} \mathfrak{X}$ is a closed immersion.
\end{definition} 

We shall denote by $\RStck/\mathfrak{S}$ the subcategory of $\Stck/\mathfrak{S}$ whose objects are algebraic stacks of finite type over $k$, representable and separated over $\mathcal{S}$, and morphisms are representable and separated morphisms of algebraic stacks, and we call $\RStck/\mathfrak{S}$ the \textit{category of representable stacks}. In the case of $\mathfrak{S} = \Spec k$, the subcategory $\RStck/\Spec k$ is just the category of separated algebraic spaces of finite type over $k$. Thus, in general, $\RStck/\mathfrak{S}$ is a significantly smaller category than $\Stck/\mathfrak{S}$.

In this section, we use the following properties of representable and separated morphisms.

\begin{lemma}\label{lemma:representable_morphisms_stacks}
The following properties hold.
    \begin{enumerate}
        \item[(i)] Representable (resp.\ separated) morphisms are closed under composition. 
        \item[(ii)] Representable (resp.\ separated) morphisms are closed under base-change: if $f \colon \mathfrak{X} \to \mathfrak{S}$ is a representable (resp.\ separated) morphism of algebraic stacks and $g \colon \mathfrak{Y} \to \mathfrak{S}$ is any morphism of algebraic stacks, then the induced morphism $\mathfrak{X} \times_{\mathfrak{S}} \mathfrak{Y} \to \mathfrak{Y}$ is representable (resp. separated).
        \item[(iii)] Let $G$ be an algebraic group over $k$, let $X$ and $Y$ be schemes over $k$ with an action of $G$, and let $f \colon X \to Y$ be an $G$-equivariant morphism. Then, the induced morphism $[X/G] \to [Y/G]$ of quotient stacks is representable. Moreover, if $f$ is separated, then the induced map is separated as well.
    \end{enumerate}
\end{lemma}

\begin{proof}
    The proofs of the first two statements follow easily from the definition. The third statement can be proven using \cite[\href{https://stacks.math.columbia.edu/tag/04ZP}{Tag 04ZP}]{stacks-project} and that $[X/G] \times_{[Y/G]} Y \simeq X$, with the morphism $[X/G] \to [Y/G]$ induced by $f$. Using this chart, separatedness follows from separatedness of $f$. 
\end{proof}

We are ready to define the Grothendieck ring of representable stacks.

\begin{definition}\label{def:grothendieck_stack}
    Let $\mathfrak{S}$ be an algebraic stack of finite type over a field $k$. The \textit{Grothendieck ring of representable stacks over $\mathfrak{S}$}, denoted by $\K(\RStck/\mathfrak{S})$, is the abelian group generated by the isomorphism classes $[\mathfrak{X}]$ of objects $\mathfrak{X}$ of $\RStck/\mathfrak{S}$, modulo the \textit{scissor relations}
    \[ [\mathfrak{X}] = [\mathfrak{Z}] + [\mathfrak{X} \setminus \mathfrak{Z}], \]
    for every closed substack $\mathfrak{Z} \subset \mathfrak{X}$ with open complement $\mathfrak{X} \setminus \mathfrak{Z}$. Note that $\mathfrak{Z}$ and $\mathfrak{X}\setminus \mathfrak{Z}$ are considered as stacks over $\mathfrak{S}$ via $\mathfrak{X}$. 
    Multiplication is given by the fiber product
    \[ [\mathfrak{X}] \cdot [\mathfrak{Y}] = [ \mathfrak{X} \times_{\mathfrak{S}} \mathfrak{Y} ] , \]
    for any algebraic stacks $\mathfrak{X}$ and $\mathfrak{Y}$.
    It is straightforward to check that this indeed gives $\K(\RStck/\mathfrak{S})$ a ring structure with unit $[(\mathfrak{S}, \id_{\mathfrak{S}})]$ and zero element $[\varnothing]$.
\end{definition}

\begin{remark}
    Notice that $2$-morphisms in $\RStck/\mathfrak{S}$ play no role in the aforementioned construction of the Grothendieck ring: two objects $(\mathfrak{X}, \pi)$ and $(\mathfrak{X}', \pi')$ are isomorphic if there exists an invertible $1$-morphism $f: (\mathfrak{X}, \pi) \to (\mathfrak{X}', \pi')$ of $\RStck/\mathfrak{S}$.
\end{remark}

\begin{remark}
    Since open and closed immersions are representable and separated morphisms, the compositions
    \[ \mathfrak{Z} \to \mathfrak{X} \to \mathfrak{S} \quad \text{ and } \quad \mathfrak{X} \setminus \mathfrak{Z} \to \mathfrak{X} \to \mathfrak{S} \]
    are indeed representable and separated by Lemma \ref{lemma:representable_morphisms_stacks}, as well as the induced morphism from the fiber product
    \[ \mathfrak{X} \times_{\mathfrak{S}} \mathfrak{Y} \to \mathfrak{S} . \]
\end{remark}

\begin{remark}
    Whenever $\mathfrak{S} = \Spec k$, we shall simply denote $\RStck/\mathfrak{S}$ by $\RStck_k$ and similarly $\K(\RStck/\mathfrak{S})$ by $\K(\RStck_k)$. 
\end{remark}

\begin{remark}
    In the usual Grothendieck ring of algebraic varieties, the relation $[E] = [\AA^n \times X]$ holds for vector bundles $E \to X$ of rank $n$. However, this is not automatic for stacks. Therefore, relations of the form 
    \[ [\mathfrak{E}] = [ \AA^n_k \times \mathfrak{X} ] \]
    for vector bundles $\mathfrak{E} \to \mathfrak{X}$ of rank $n$ are usually added in the definition of the Grothendieck ring of stacks \cite{joyce2007motivic, bri12, eke09, behdhi07}. We omit this assumption in our definition and we work only with the scissor relations. This is crucial for us since, in this paper, we will work with $\K(\RStck/\BG)$, the Grothendieck ring of representable stacks over the classifying space $\BG = [\Spec k / G]$, and we want to remember the group action on the fibres. 
\end{remark}

\begin{remark}\label{rem:modulestructure}
    Any morphism $\mathfrak{X} \to \mathfrak{S}$ of algebraic stacks induces a $\K(\RStck/\mathfrak{S})$-module structure on $\K(\RStck/\mathfrak{X})$, where the module structure is given on the generators by
\[ [\mathfrak{T}] \cdot [\mathfrak{Y}] = [\mathfrak{T} \times_\mathfrak{S} \mathfrak{Y}],\]
for representable and separated morphisms $\mathfrak{T}\to \mathfrak{S}$ and $\mathfrak{Y}\to \mathfrak{X}$ of algebraic stacks. Observe that the composite map $\mathfrak{T} \times_\mathfrak{S} \mathfrak{Y}\to \mathfrak{Y}\to \mathfrak{X}$ is representable and separated by Lemma \ref{lemma:representable_morphisms_stacks} as both morphisms 
$\mathfrak{T} \times_\mathfrak{S} \mathfrak{Y}\to \mathfrak{Y}$ and $\mathfrak{Y}\to \mathfrak{X}$ are representable and separated. 
\end{remark}


A representable and separated morphism of algebraic stacks $f \colon \mathfrak{X} \to \mathfrak{Y}$ over $\mathfrak{S}$ induces a functor 
\[f_!: \RStck/\mathfrak{X}\to \RStck/\mathfrak{Y}\]
given by composing with $f$. Indeed, if $g:\mathfrak{T}\to \mathfrak{X}$ is representable and separated, then $f\circ g:\mathfrak{T}\to \mathfrak{Y}$ is representable and separated by Lemma \ref{lemma:representable_morphisms_stacks}. It is straightforward that this functor induces a $\K(\Stck/\mathfrak{S})$-module morphism
\[\K(\RStck/\mathfrak{X}) \to \K(\RStck/\mathfrak{Y})\]
which we will denote by $f_!$ as well.
Similarly, any morphism of algebraic stacks $f \colon \mathfrak{X} \to \mathfrak{Y}$ over $\mathfrak{S}$ induces a functor
\[f^*: \RStck/\mathfrak{Y} \to \RStck/\mathfrak{X}\]
given by pulling back along $f$. Indeed, Lemma \ref{lemma:representable_morphisms_stacks} shows that if
$g:\mathfrak{T}\to \mathfrak{Y}$ is representable and separated, then the map $\mathfrak{T}\times_{\mathfrak{Y}}\times \mathfrak{X}\to \mathfrak{X}$ given by the fiber product is also representable and separated. It is easy to see that this functor induces a $\K(\Stck/\mathfrak{S})$-module morphism 
\[ f^* \colon \K(\RStck/\mathfrak{Y}) \to \K(\RStck/\mathfrak{X}).\]
The morphism $f^*$ is a ring homomorphism, making $\K(\RStck/\mathfrak{X})$ into a $\K(\RStck/\mathfrak{Y})$-algebra. However, note that $f_!$ is not a ring morphism, since generally it does not send units to units (indeed, the behaviour of $f_!$ with respect to the ring structure is given by the pull-push formula).

\section[Representable stacks over BG]{Representable stacks over $\BG$}
\label{sec:equivariant-setting}


Fix an algebraic group $G$ over a field $k$, and consider its classifying stack $\BG = [\Spec k / G]$. In this case, the category $\RStck/\BG$ is equivalent to the category of separated $G$-algebraic spaces of finite type over $k$, as shown in the following result.

\begin{lemma}\label{lem:quotient-G-torsor}
    The functor
    \[ [- / G] \colon G\textup{-}\Spc\to \RStck/\BG, \quad X \mapsto [X/G] \]
    is an equivalence of categories, where $G\textup{-}\Spc$ is the category whose objects are separated algebraic spaces of finite type over $k$ equipped with a $G$-action, and whose morphisms are $G$-equivariant morphisms of algebraic spaces.
\end{lemma}

\begin{proof}
    We describe the inverse of the functor above. Write $\point = \Spec k$ and consider the quotient map $c \colon \point \to \BG$ given by the trivial $G$-torsor $G \to \point$. Let $\mathfrak{X} \to \BG$ be a representable and separated morphism of algebraic stacks. Then, the functor $c^* \colon \RStck/\BG \to \RStck_k$ sends $\mathfrak{X}$ to $\mathfrak{X}\times_\BG \point$ which is a separated algebraic space by definition. We denote this algebraic space by $X$, and notice that the map $X \to \mathfrak{X}$ induced by the fiber product shows that $X$ is a $G$-torsor over $\mathfrak{X}$.
    
    To prove that $c^* \circ [- / G] \simeq \id$, consider the following commutative diagram with the obvious maps
    \[ \begin{tikzcd} X \arrow{r} \arrow{d} & \point \arrow{d} \\ {[X/G]} \arrow{r} & \BG , \end{tikzcd} \]
    which induces a morphism of $G$-torsors $X \to [X/G] \times_{BG} \star$. Since any morphism of $G$-torsors is an isomorphism, we have $X \simeq [X/G] \times_{BG} \star$ as desired.
    
    To prove that $[-/G] \circ c^* \simeq \id$, observe that for any scheme $U$, the objects of $[X/G](U)$ are given by diagrams
    \[ \begin{tikzcd} P \arrow{r}{f} \arrow[swap]{d}{\pi} & X \\ U & \end{tikzcd} \]
    where $P \xrightarrow{\pi} U$ is a principal $G$-bundle and $f$ is a $G$-equivariant map. Since $X$ is a $G$-torsor over $\mathfrak{X}$, the morphism $f$ descends to a morphism $\overline{f} \colon U \to \mathfrak{X}$ such that $\overline{f} \circ \pi = \pi \circ f$. Conversely, for any morphism $U \to \mathfrak{X}$, the pullback $U \times_\mathfrak{X} X$ is a principal bundle equipped with an equivariant map to $X$. Hence $[X/G](U) \simeq \textrm{Hom}(U, \mathfrak{X}) = \mathfrak{X}(U)$ naturally for all schemes $U$ and thus $\mathfrak{X} \simeq [X/G]$.
\end{proof}

A direct consequence of the above proof is the following useful characterization of $\RStck/\BG$.

\begin{corollary}
    Every object in $\RStck/\BG$ is isomorphic to a $G$-quotient stack of a separated algebraic space of finite type over $k$.
\end{corollary}

\begin{remark}
    Let $\mathfrak{X} = [X/G]$ and $\mathfrak{Y} = [Y/G]$ be quotient stacks over $\BG = [\point/G]$ with $X$ and $Y$ algebraic spaces of finite type over $k$. Then, the fibre product is also a global quotient stack given as
    \[ \mathfrak{X} \times_\BG \mathfrak{Y} = [(X \times_k Y) / G]. \]
    This provides a simple description of the multiplication structure of the ring $\K(\RStck/\BG)$. 
\end{remark}

Let $\mathfrak{X} = [X/G] \to \BG$ be a representable and separated morphism, and let $H \subset G$ be any algebraic subgroup. We consider the fixed point stack $X^H$ with its natural morphism $X^H\to X$. The fixed point stack $X^H$ is an algebraic space, since $X^H\to X$ is a representable morphism of algebraic stacks \cite[Theorem 3.3]{romagny2005group}. The following lemma, which is the generalization of a result of Fogarty's \cite{fogarty1973fixed}, was communicated to us by Matthieu Romagny.


\begin{lemma}\label{le:romagny}
    Under the above assumptions, the natural map of algebraic spaces $X^H\to X$ is a closed immersion.
\end{lemma}

\begin{proof}

Consider the fiber product of the diagonal map $\Delta:X\to X\times X$ and the action map $H\times X\to X\times X$ (given by $(h,x)\mapsto (hx,x)$)
 \[ \begin{tikzcd} Y \arrow{r} \arrow[swap]{d}{f} & X\arrow{d}{\Delta} \\ H\times X \arrow{r} & X\times X \end{tikzcd}.\]
Since, the diagonal map $X\to X\times X$ is a closed immersion, the map $Y\to H\times X$ is also a closed immersion.

We regard $H\times X$ as an algebraic space over $X$ using the second projection. The fixed point locus, $X^H$, is the largest subfunctor of $X$ so that $f$ becomes an isomorphism under the base change $X^H\to X$. As a result, it is the Weil restriction of $Y$ under the map $H\times X\to X$. This is representable by a closed subspace of $X$ by a version of Proposition B.3 for algebraic spaces in \cite{abramovich2012moduli}.
\end{proof}

This yields a $\K(\RStck_k)$-module morphism
\[ (-)^H \colon \K(\RStck/\BG) \to \K(\RStck_k) , \]
sending the class of a representable morphism $\mathfrak{X} \to \BG$ to the class of $X^H$ in $\K(\RStck_k)$. 

Several considerations are in order. Again, denote by $c \colon \star \to \BG$ the quotient map given by the trivial $G$-bundle on $\star$. 
\begin{itemize}
    \item The map $X^H \to \star$ as an element of $\Stck_k$ is induced by the composition $X^H \to X \stackrel{c}{\to} \star$ (cf. Lemma \ref{lem:quotient-G-torsor}). In particular, it is a representable and separated morphism.
    \item We view $\K(\RStck/\BG)$ as a $\K(\RStck_k)$-module via the map $\K(\RStck_k) \to \K(\RStck/\BG)$ induced by $c_!$. In other words, we equip every space with the trivial $G$-action.
\end{itemize}


\begin{example}\label{ex:gln}
Let $G = \text{GL}_n(k)$ act on itself by conjugation, and $\mathfrak{G} = [G/G]$ the corresponding quotient stack. Then the points of $G$ fixed under the action of $G$ is the center $Z(G) = \{ \lambda \cdot I \colon \lambda \in k^* \}$, so
    \[ \mathfrak{G}^G = [Z(G)] = q - 1 \in \K(\RStck_k) , \]
where $q$ denotes the class of the affine line $\AA^1_k$.
\end{example}

\begin{example}
Let $G$ be a finite group acting on a variety $X$ over a field $k$ of characteristic coprime to the order of the group $G$. Then, the orbifold Euler characteristic \cite{atiyah1989equivariant, hirzebruch1990euler} is defined as
\[ \chi^{\orb}(X,G) = \frac{1}{|G|} \sum_{\substack{g_1, g_2 \in G \\ [g_1, g_2] = 1}} \chi(X^{\langle g_1, g_2 \rangle}) , \]
where $X^{\langle g_1, g_2 \rangle}$ denotes the locus fixed by both $g_1$ and $g_2$. The orbifold Euler characteristic can be lifted to a $\K(\Stck_k)$-module map \cite{gusein2019grothendieck}
\[ f = \sum_{\substack{g_1, g_2 \in G \\ [g_1, g_2] = 1}} \left[ (-)^{\langle g_1, g_2\rangle} \right] \colon \K(\Stck/\BG)\to \K(\Stck_k) \]
making the following diagram commute:
\[ \begin{tikzcd}
    \K(\RStck/\BG) \arrow{r}{f} \arrow[swap]{d}{\chi^\textup{orb}} & \K(\RStck_k) \arrow{d}{\frac{1}{|G|} \chi} \\
    \ZZ[|G|^{-1}] \arrow[equals]{r} & \ZZ[|G|^{-1}].
\end{tikzcd} \]
\end{example}


\begin{remark}
    Ekedahl \cite{eke09} defines a Grothendieck ring of stacks $\widetilde{\K}(\Stck_k)$ as the abelian group generated by stacks of finite type over $k$ with affine stabilizers module the scissor relations and the additional relation that for rank $n$ vector bundles $\mathfrak{E} \to \mathfrak{X}$ we impose
    \[ [\mathfrak{E}] = [ \AA^n_k \times \mathfrak{X} ]. \]
    Ekedahl shows that $\widetilde{\K}(\Stck_k)$ is isomorphic to $\hat{\K}(\Var_k)$ which is the localization of the Grothendieck ring of varieties, $\K(\Var_{k})$, by inverting the class of the affine line $q = [\AA_k^1]$ and the classes of the form $q^n-1$. 
    
    In the case of an affine algebraic group $G$, and a representable morphism of stacks, $\mathfrak{X}\to \BG$, the stack $\mathfrak{X}$ has affine stabilizers. Thus, we obtain a natural map
    \[\K(\RStck/\BG)\to \widetilde{\K}(\Stck_k)\]
    by forgetting the map to $\BG$. Composing this map with the isomorphism $\widetilde{\K}(\Stck_k)\to \hat{\K}(\Var_{k})$, we obtain a map
    \begin{equation}\label{eq:evaluation} 
        \textup{ev} \colon \K(\RStck/\BG)\to \hat{\K}(\Var_k)
    \end{equation}
    which we call the \textit{evaluation map}.
    
    Alternatively, we can define the evaluation map in another, equivalent way using stratifications. First, the quotient stack can be stratified $\mathfrak{X} = [X/G]$ by locally closed substacks of the form $[X_i/\GL_{n_i}(k)]$ where the $X_i$ are quasi-projective schemes \cite{kres99}. Then, it can be shown that the class $\sum_i [X_i]/[\GL_{n_i}(k)]\in \hat{\K}(\Var_k)$ does not depend on the stratification \cite{behdhi07}. 
\end{remark}

\begin{remark}     
    This evaluation map has a simple form if the group $G$ is \textit{special}. Recall that an algebraic group $G$ is called \textit{special} if any $G$-torsor is Zariski-locally trivial, in other words, if $\mathfrak{E}\to \mathfrak{X}$ is a $G$-torsor, then the relation $[\mathfrak{E}]=[G][\mathfrak{X}]$ holds in the Grothendieck-ring of stacks $\widetilde{\K}(\Stck_k)$. Applying this to the torsor $\point\to \BG$, we have the relation $[\BG]=[G]^{-1}$. Similarly, the map described in \eqref{eq:evaluation} sends the class of a global quotient stack $[X/G]$ to $[X]/[G]$ for any special algebraic group $G$. Special algebraic groups include $\GL_n(k), \SL_n(k)$ and $\AGL_n(k)$.
    
\end{remark}

\begin{example}
    Consider two stacks, $\mathfrak{X} = [G/G]$ with $G=\GL_n(k)$ acting on itself by left translation, and $\mathfrak{Y} = [G/G]$ with $G = \GL_n(k)$ acting on itself by conjugation as in Example \ref{ex:gln}. Since $G$ is special, these stacks have the same classes under the evaluation map $\K(\RStck/\BG) \to \hat{\K}(\Var_k)$, namely the unit. However, the classes of $\mathfrak{X}$ and $\mathfrak{Y}$ are different in $\K(\RStck/\BG)$ as their images under the map $(-)^G \colon \K(\RStck/\BG)\to \K(\RStck_k)$ are different ($0$ and $q-1$ respectively).
\end{example}

\section{Constructing the stacky TQFT}
\label{sec:stacky_TQFT}
In this section, we follow \cite{thesisangel, gon20, gonlogmun20, habvog20, thesisvogel} to construct a Topological Quantum Field Theory (TQFT) which computes the classes of character stacks in the Grothendieck ring of stacks $\K(\RStck/\BG)$.

\subsection{The category of bordisms}\label{sec:cat-bordisms}
In this section, we follow closely \cite{milnor} and \cite{kock} in defining the category of bordisms. Throughout the paper, a manifold is always assumed to be smooth.

\begin{definition}
    A \textit{bordism} between two $(n - 1)$-dimensional closed manifolds $M_1$ and $M_2$, is an $n$-dimensional manifold $W$ (with boundary) with maps
    \[ \begin{tikzcd} M_2 \arrow{r}{i_2} & W & \arrow[swap]{l}{i_1} M_1 \end{tikzcd} \]
    where $\partial W = i_1(M_1) \sqcup i_2(M_2)$. Two such bordisms $W, W'$ are \textit{equivalent} if there exists a diffeomorphism $F: W \xrightarrow{\sim} W'$ such that the following diagram
    \[ \begin{tikzcd}[row sep=0.5em] & W \arrow{dd}{F} & \\ M_2 \arrow{ur} \arrow{dr} & & M_1 \arrow{ul} \arrow{dl} \\ & W' & \end{tikzcd} \]
    commutes.
\end{definition}


Given two bordisms $W : M_1 \to M_2$ and $W' : M_2 \to M_3$, one obtains a new bordism $W' \circ W: M_1\to M_3$ by gluing $W$ and $W'$ along the images of $M_1$ \cite{milnor}. Note that the gluing of bordisms is well-defined up to diffeomorphism. For this reason, we only consider equivalence classes of bordisms.

\begin{definition}
    The \textit{category of $n$-bordisms}, denoted $\Bd_n$, is defined as the category whose objects are $(n - 1)$-dimensional closed manifolds, and its morphisms $M_1 \to M_2$ are equivalence classes of bordisms from $M_1$ to $M_2$. Composition is given by the above gluing.
\end{definition}

The definition of representation varieties and character varieties involve the fundamental group of manifolds. We alter the definition above by considering points on the manifolds.

\begin{definition}
    The \textit{category of $n$-bordisms with basepoints}, denoted $\Bdp_n$, is the category consisting of:
    \begin{itemize}
        \item Objects: pairs $(M, A)$ with $M$ being an $(n - 1)$-dimensional closed manifold, and $A \subset M$ a finite set of points intersecting each connected component of $M$.
        
        \item Morphisms: a map $(M_1, A_1) \to (M_2, A_2)$ is given by a class of pairs $(W, A)$ with $W : M_1 \to M_2$ a bordism, and $A \subset W$ a finite set intersecting each connected component of $W$ such that $A \cap M_1 = A_1$ and $A \cap M_2 = A_2$. Two such pairs $(W, A)$ and $(W', A')$ are equivalent if there is a diffeomorphism $F : W \to W'$ such that $F(A) = A'$ and such that the diagram
        \begin{equation} \label{eq:diagram_from_definition_bdp} \begin{tikzcd}[row sep=0.5em] & W \arrow{dd}{F} & \\ M_2 \arrow{ur} \arrow{dr} & & M_1 \arrow{ul} \arrow{dl} \\ & W' & \end{tikzcd} \end{equation}
        commutes.
        
        The composition is again obtained by gluing the bordisms and the marked points.
    \end{itemize}
    In order to have identity morphism for the objects $(M, A)$, we allow $(M, A)$ itself to be considered as a bordism $(M, A) \to (M,A)$.
\end{definition}

\subsection{Stacky TQFT}

In this section, we define $G$-representation varieties and the $G$-character stacks associated to pairs $(X,A)$ where $X$ is a compact connected manifold and $A\subset X$ is a finite set of basepoints. The main result of this section is Theorem \ref{thm:tqft-stacky} that generalizes the construction of \cite{thesisangel, gon20, gonlogmun20} to compute the virtual class of $G$-character stacks.

\begin{definition}
    Let $(X, A)$ be a pair of topological spaces. The \textit{fundamental groupoid of $X$ with respect to $A$}, denoted $\Pi(X, A)$, is the groupoid category whose objects are elements of $A$, and an arrow $a \to b$ for each homotopy class of paths from $a$ to $b$. Composition of morphisms is given by concatenation of paths. In particular, if $A = \{ x_0 \}$ is a single point, we obtain the fundamental group $\pi_1(X, x_0)$ as the group of endomorphisms $\Pi(X,\{x_0\})_{x_0}$ of the object $x_0$.
\end{definition}

Suppose that $A \subset X$ is a finite set and, for each connected component of $X$, let us pick exactly one element of $A$ contained in it, obtaining a subset $S = \{ a_1, \ldots, a_s \} \subset A$. For any other element $a$ of $A$, we pick a morphism $f_a : a_i \to a$ for the point $a_i \in A$ which is in the same connected component as $a$. It is easy to see that any morphism of groupoids $\rho : \Pi(X,A) \to \mathcal{G}$ from the fundamental groupoid to the groupoid $\mathcal{G}$ associated to the group $G$ (i.e., the groupoid with a single object whose morphism group is the group $G$) is uniquely determined by the group homomorphisms $\rho_i : \pi_1(X,a_i) \to G$ and the choices of $\rho(f_a) \in G$. Thus, we have
\begin{equation}
    \label{eq:the_hom_equation}
    \Hom_{\Grpd}(\Pi(X,A), \mathcal{G}) \simeq \Hom(\pi(X,a_1), G) \times \cdots \times \Hom(\pi_1(X,a_s), G) \times G^{|A| - s} .
\end{equation}
If $G$ is an algebraic group, each of these factors naturally carries the structure of an algebraic variety, and this structure is independent on the choices.

\begin{definition}
    Let $X$ be a compact connected manifold (possibly with boundary), $A \subset X$ a finite set of basepoints and $G$ an algebraic group. Then the \textit{$G$-representation variety} of the pair $(X, A)$ is defined as the set of functors
    \[ R_G(X, A) = \Hom_{\Grpd}(\Pi(X, A), \mathcal{G}) . \]
    The set above has a structure of a variety, in fact, it can be identified with a closed subvariety of $G^n$ for some $n$. Note that $G$ acts on $R_G(X,A)$ by conjugation, and the corresponding global quotient stack
    \[ \mathfrak{X}_G(X, A) = [ R_G(X, A) / G ]  \]
    is called the \textit{$G$-character stack}.
\end{definition}

The conjugation action of $G$ on the representation variety $R_G(X, A)$ acts component-wise on the factors described in (\ref{eq:the_hom_equation}). As a result, the character stack $\mathfrak{X}_G(X,A)$ has a similar decomposition
\begin{equation}
    \label{eq:the_hom_equation2}
    \mathfrak{X}_G(X,A) \simeq [R_G(X,S) \times G^{|A| - s}/G] = \mathfrak{X}_G(X,S) \times_\BG [G/G]^{|A|-s} .
\end{equation}

\begin{remark}\label{rmk:character-variety}
There is a different approach to the quotient of the representation variety under the adjoint action of $G$. Suppose that $G$ is affine and let $S$ be the ring of regular functions on $R_G(X,A)$, so that $R_G(X,A) = \Spec S$. The action of $G$ on $R_G(X,A)$ induces an action on $S$. If, in addition, $G$ is a reductive group, then by Nagata's theorem \cite{nagata1963invariants} we have that the $G$-invariant elements of $S$, $S^G$, is a finitely generated $k$-algebra. In this way, we define the GIT quotient of $R_G(X,A)$ under $G$ as
\[ R_G(X,A) \sslash G = \Spec S^G. \]
This is an affine algebraic variety called the \textit{$G$-character variety}. Notice that this variety contains, in general, less information than the $G$-character stack: while the latter keeps all the orbits of the $G$ action, in the former some orbits are collapsed in the GIT quotient. In Section \ref{sec:AGL1-character-stack} we will compute the virtual class of the character variety from the one of the character stack in the case of $\AGL_1(k)$.
\end{remark}

Given an algebraic group $G$, we can consider the category $\Span(\RStck/\BG)$ of spans of $\BG$-stacks. At the level of objects, this category has the same objects as $\RStck/\BG$, namely, representable stacks over $\BG$. H
The morphisms in $\Span(\RStck/\BG)$ between two $\BG$-stacks $\mathfrak{X} \to \mathfrak{Y}$ is an equivalence class of triples $(\mathfrak{Z}, f, g)$ where $\mathfrak{Z}$ is a $\BG$-stack and $f$ and $g$ are morphisms
\[ \begin{tikzcd} \mathfrak{X} & \arrow[swap]{l}{f} \mathfrak{Z} \arrow{r}{g} & \mathfrak{Y} \end{tikzcd} \]
The equivalence relation is given as follows:\ two triples $(\mathfrak{Z}, f, g)$ and $(\mathfrak{Z}', f', g')$ are declared as equivalent if there exists an isomorphism $\alpha: \mathfrak{Z} \to \mathfrak{Z}'$ such that the following diagram 2-commutes
\begin{equation}
    \label{eq:equivalence-spans}
    \begin{tikzcd}[row sep=0.25em]
        & \mathfrak{Z} \arrow[swap]{ld}{f} \arrow{rd}{g} \arrow{dd}{\alpha} & \\ \mathfrak{X} & & \mathfrak{Y} \\ & \mathfrak{Z}' \arrow{lu}{f'} \arrow[swap]{ru}{g'} & 
    \end{tikzcd}
\end{equation}
Composition in this category is given by fibered product. Explicitly, if $(\mathfrak{Z}_1, f_1, g_1): \mathfrak{X} \to \mathfrak{Y}$ and $(\mathfrak{Z}_2, f_2, g_2): \mathfrak{Y} \to \mathfrak{Z}$ are two morphisms, then its composition is
    \[\xymatrix @R=.5pc{
        & & \mathfrak{Z}_1 \ar[rd] \ar[ld] \times_{\mathfrak{Y}} \mathfrak{Z}_2 & & \\
        & \mathfrak{Z}_1 \ar[rd]^{g_1} \ar[ld]_{f_1} & & \mathfrak{Z_2} \ar[rd]^{g_2} \ar[ld]_{f_2} & \\
        \mathfrak{X} & & \mathfrak{Y} & & \mathfrak{Z} \\
        }
    \]
Notice that such fibered product is well-defined up to equivalence of triples. Finally, $\Span(\RStck/\BG)$ inherits the monoidal structure from $\RStck/\BG$ in a natural way: $\mathfrak{X} \otimes \mathfrak{Y} = \mathfrak{X} \times_{\BG} \mathfrak{Y}$ on objects and $(\mathfrak{Z}_1, f_1, g_1) \otimes (\mathfrak{Z}_2, f_2, g_2) = (\mathfrak{Z}_1 \times_{\BG} \mathfrak{Z}_2, f_1 \times f_2, g_1 \times g_2)$ on morphisms.

\begin{remark}
The span category $\Span(\RStck/\BG)$ has indeed a natural bicategory structure by taking $1$-morphisms as triples on the nose and as $2$-morphisms morphisms of triples as in (\ref{eq:equivalence-spans}). From this point of view, the (standard) category structure we have defined here is nothing but the truncation 
of this bicategory structure. In this vein, most of the constructions described in this paper can be straightforwardly extended to the bicategory setting. However, we will not follow this approach here since the usual category structures will be enough for the purposes of this work. For a more detailed account of the TQFTs in the bicategory setting, we refer the reader to \cite{gonzalez2023arithmetic}.
\end{remark}

Using this auxiliary category, let us construct a monoidal functor
\[ \mathcal{F} : \Bdp_n \to \Span(\RStck/\BG) , \]
from the category of bordisms to the category of spans over $\RStck/\BG$ by sending an object $(M, A)$ to $\mathfrak{X}_G(M, A)$ and a bordism $(W, A) : (M_1, A_1) \to (M_2, A_2)$ to the span
\[ \mathfrak{X}_G(M_1, A_1) \leftarrow \mathfrak{X}_G(W, A) \rightarrow \mathfrak{X}_G(M_2, A_2) , \]
whose maps are induced from the inclusions $(M_i, A_i) \to (W, A)$. Notice that the maps $R_G(W, A) \to R_G(M_i, A_i)$ descend to the quotient stack since the restriction maps are $G$-equivariant for the conjugacy action. Moreover, by Lemma \ref{lemma:representable_morphisms_stacks} (iii), the morphisms $\mathfrak{X}_G(W,A) = [R_G(W, A)/G] \to \mathfrak{X}_G(M_i, A_i) = [R_G(M_i, A_i)/G]$ are representable. The assignment $\mathcal{F}$ will be referred to as the \textit{field theory}.

\begin{proposition}
The assignment $\mathcal{F}: \Bdp_n \to \Span(\RStck/\BG)$ is a monoidal functor.

\begin{proof}
    Suppose that we have two bordisms $(W, A) : (M_1, A_1) \to (M_2, A_2)$ and $(W', A') : (M_2, A_2) \to (M_3, A_3)$. Consider small collarings $U \cong X_2 \times [0,1) \subset W$ and $U' \cong X_2 \times [0,1) \subset W'$ around the boundary $X_2$ in $W$ and $W'$, respectively, such that $U \cap A = A_2$ and $U' \cap A' = A_2$. Then $U \cup_{X_2} U'$ is an open set of $W \cup_{X_2} W'$ with the pair $(U \cup_{X_2} U', A_2)$ homotopically equivalent to $(X_2, A_2)$, $U \cup_{X_2} W'$ is an open set of $W \cup_{X_2} W'$ with $(U \cup_{X_2} W', A')$ homotopically equivalent to $(W', A')$, and $W \cup_{X_2} U'$ is an open set of $W \cup_{X_2} W'$ with $(W \cup_{X_2} U', A)$ homotopically equivalent to $(W, A)$.
    
    Therefore, by the Seifert-van Kampen theorem for fundamental groupoids \cite{brown1967groupoids}, we get a co-cartesian square
    \[\xymatrix{
        \Pi(W \cup_{X_2} W', A \cup A') & \ar[l] \Pi(W, A) \\
        \Pi(W', A') \ar[u] & \ar[l]\ar[u] \Pi(X_2, A_2)
        }
    \]
    Since the functor $\Hom_{\Grpd}(-, \mathcal{G})$ is continuous, this implies that $R_G(W \cup_{X_2} W', A \cup A')$ coincides with the pullback $R_G(W,A) \times_{R_G(X_2,A_2)} R_G(W',A')$. Moreover, the projection maps are $G$-equivariant morphisms for the adjoint action of $G$, so taking the stacky quotient gives rise to a pullback diagram
    \begin{equation}\label{eq:cartesian-square}
    \begin{split}
        \xymatrix{
        \mathfrak{X}_G(W \cup_{X_2} W', A \cup A') \ar[d]\ar[r] & \ar[d] \mathfrak{X}_G(W,A) \\
        \mathfrak{X}_G(W',A') \ar[r] & \mathfrak{X}_G(X_2, A_2)
        }
    \end{split}
    \end{equation}
    
    This means that the composition of the two spans induced by the field theory functor
    $$
    \mathfrak{X}_G(M_1, A_1) \leftarrow \mathfrak{X}_G(W, A) \rightarrow \mathfrak{X}_G(M_2, A_2), \quad \mathfrak{X}_G(M_2, A_2) \leftarrow \mathfrak{X}_G(W', A') \rightarrow \mathfrak{X}_G(M_3, A_3)
    $$
    is the composed span, as shown by the following diagram whose middle diamond is the cartesian square (\ref{eq:cartesian-square}).
    \[\xymatrix @R=.8pc @C=.8pc{
        & & \mathfrak{X}_G(W \cup_{X_2} W', A \cup A') \ar[dr] \ar[dl] & & \\
        & \mathfrak{X}_G(W,A) \ar[dr] \ar[dl] & & \mathfrak{X}_G(W',A') \ar[dr] \ar[dl] & \\
        \mathfrak{X}_G(M_1, A_1) & & \mathfrak{X}_G(M_2, A_2) & & \mathfrak{X}_G(M_3, A_3)
        }
    \]
    The monoidality of the field theory functor $\cF$ is obvious, since $\mathfrak{X}_G(M \sqcup M', A \sqcup A') = \mathfrak{X}_G(M, A) \times_{\BG} \mathfrak{X}_G(M',A')$.
\end{proof}
\end{proposition}

Next, we construct the \textit{quantization functor}
\[ \mathcal{Q} : \Span(\RStck/\BG) \to \K(\RStck/\BG)\text{-}\Mod \]
by assigning to a stack $\mathfrak{X}$ the $\K(\RStck/\BG)$-module $\K(\RStck/\mathfrak{X})$, and to a span $(\mathfrak{Z}, f, g) = \left(\mathfrak{X} \xleftarrow{f} \mathfrak{Z} \xrightarrow{g} \mathfrak{Y}\right)$ the morphism $g_! \circ f^* : \K(\RStck/\mathfrak{X}) \to \K(\RStck/\mathfrak{Y})$. Recall again that, by Lemma \ref{lemma:representable_morphisms_stacks}, the maps $g_!$ and $f^*$ send representable morphisms to representable morphisms. 

Furthermore, notice that this homomorphism $g_! \circ f^* : \K(\RStck/\mathfrak{X}) \to \K(\RStck/\mathfrak{Y})$ does not depend on the representative chosen for the equivalence class of the triple $(\mathfrak{Z}, f, g)$. Indeed, if $(\mathfrak{Z}', f', g')$ is an equivalent triple related through an isomorphism $\alpha: \mathfrak{Z} \to \mathfrak{Z'}$, since $f = f' \circ \alpha$ and $g = g \circ \alpha$, then $g_! \circ f^* = (g')_! \circ \alpha_! \circ \alpha^* \circ (f')^* = (g')_! \circ (f')^*$. Here, we have used that $\alpha_! \circ \alpha^* = \id$ since the following diagram is a cartesian square
\[
    \xymatrix{
        \mathfrak{Z} \ar[r]^{\alpha} \ar[d]_{\alpha} & \mathfrak{Z}' \ar[d]^{\id} \\
        \mathfrak{Z}' \ar[r]_{\id}  & \mathfrak{Z}' \\
        }
\]

To prove that the quantization $\mathcal{Q}$ is actually a functor we need the following auxiliary result of cartesian categories.

\begin{lemma}\label{lem:beck-chevalley}
Let $\cC$ be a category with pullbacks. For any objects $A, X, Y, Z \in \cC$ equipped with morphisms $A \to X$, $X \to Z$ and $Y \to Z$, we have an isomorphism
$$
    A \times_X (X \times_Z Y) \cong A \times_Z Y.
$$

\end{lemma}

As an immediate consequence, we have the following.

\begin{corollary}\label{cor:beck-chevalley}
Consider a cartesian square of representable and separable $\mathfrak{S}$-stacks 
\[
    \xymatrix{
        \mathfrak{X} \times_{\mathfrak{Z}} \mathfrak{Y} \ar[r]^{\;\;\tilde{g}}\ar[d]_{\tilde{f}} & \mathfrak{X} \ar[d]^{f} \\
        \mathfrak{Y} \ar[r]_{g} & \mathfrak{Z}
        }
\]
Then, in $\K(\RStck/\mathfrak{S})\text{-}\Mod$, we have that $g^* \circ f_! = \tilde{f}_! \circ \tilde{g}^*$.
\end{corollary}

\begin{proof}
Consider an element $\mathfrak{S}\to \mathfrak{X}$ in $\RStck/\mathfrak{X}$. Using the lemma above, we have an isomorphism between $(\tilde{f}_!\circ \tilde{g}^*)(\mathfrak{S})$ and $\mathfrak{S}\times_\mathfrak{Z} \mathfrak{Y}$. The latter is $(g^*\circ f_!)(\mathfrak{S})$ implying the statement.     
\end{proof}

\begin{remark}
In the context of sheaves over schemes, the property of Corollary \ref{cor:beck-chevalley} is usually called the Beck-Chevalley property or the base change property.
\end{remark}

\begin{proposition}
The assignment $\mathcal{Q}: \Span(\RStck/\BG) \to \K(\RStck/\BG)\text{-}\Mod$ is a lax monoidal functor.

\begin{proof}
    Let us check that $\mathcal{Q}$ preserves the composition of spans. Consider two spans of $\BG$-stacks $S_1: \mathfrak{X} \xleftarrow{f} \mathfrak{W} \xrightarrow{g} \mathfrak{Y}$ and $S_2: \mathfrak{Y} \xleftarrow{s} \mathfrak{W'} \xrightarrow{t} \mathfrak{Z}$, whose composition is given by the diagram
    \begin{equation}\label{eq:composition-span}
    \begin{split}
    \xymatrix{
        & & \mathfrak{W} \times_{\mathfrak{Y}} \mathfrak{W'} \ar[dr]^{\tilde{g}} \ar[dl]_{\tilde{s}} & & \\
        & \mathfrak{W} \ar[dr]^{g} \ar[dl]_{f} & & \mathfrak{W'} \ar[dr]^{t} \ar[dl]_{s} & \\
        \mathfrak{X} & & \mathfrak{Y} & & \mathfrak{Z}
        }
    \end{split}
    \end{equation}
    This means that, after applying the functor $\mathcal{Q}$, the resulting map is $\mathcal{Q}(S_2 \circ S_1) = (t \circ \tilde{g})_! \circ (f \circ \tilde{s})^* = t_!\tilde{g}_!\tilde{s}^*f^*$. But, by Corollary \ref{cor:beck-chevalley}, since the middle diamond of (\ref{eq:composition-span}) is cartesian, we have that $\tilde{g}_!\tilde{s}^* = s^*g_!$ and, thus
    $$
        \mathcal{Q}(S_2 \circ S_1) = t_!\tilde{g}_!\tilde{s}^*f^* = t_!s^*g_!f^* = \mathcal{Q}(S_2) \circ \mathcal{Q}(S_1).
    $$
    
    For the lax monoidality, notice that the external product defines a morphism 
    \[\boxtimes: \K(\RStck/\mathfrak{X}) \otimes_{\K(\RStck/\BG)} \K(\RStck/\mathfrak{Y}) \to \K\left(\RStck/\mathfrak{X} \times_{\BG} \mathfrak{Y}\right).\] 
    Explicitly, it is induced, for $\mathfrak{A} \in \RStck/\mathfrak{X}$ and $\mathfrak{B} \in \RStck/\mathfrak{Y}$, by the map
    $$
        \mathfrak{A} \otimes \mathfrak{B} \mapsto \mathfrak{A} \times_{\BG} \mathfrak{B}.
    $$
    This external product provides the lax monoidality of the functor $Q$.
\end{proof}
\end{proposition}

To finish the construction, we define the symmetric lax monoidal TQFT as the composition of the field theory and the quantization functor
\[ Z = \mathcal{Q} \circ \mathcal{F} : \Bdp_n \to \K(\RStck/\BG)\text{-}\Mod . \]

We can regard a closed connected manifold $X$ of dimension $n$ with a chosen base-point $\point$ on $X$ as a bordism $(X, \point) : \varnothing \to \varnothing$. In this way, $\mathcal{F}(X, \point)$ is the span
 \[ \BG=[\point/G] \overset{t}{\longleftarrow} \mathfrak{X}_G(X, \point) = \mathfrak{X}_G(X) \overset{t}{\longrightarrow} [\point/G]=\BG. \]
 Hence $Z(X, \point)(1) = t_! t^*([\BG \to \BG]) = t_!\left([ \mathfrak{X}_G(X) \to \mathfrak{X}_G(X)]\right) = [ \mathfrak{X}_G(X) \to \BG]$. Working similarly with any number of points, we have proven the following result.
 
\begin{theorem}\label{thm:tqft-stacky}
Let $G$ be an algebraic group. There exists a symmetric lax monoidal functor (i.e.\ a lax monoidal TQFT)
$$
    Z: \Bdp_n \to \K(\RStck/\BG)\text{-}\Mod
$$
computing the virtual classes in $\K(\RStck/\BG)$ of $G$-character stacks over closed manifolds.
\end{theorem}

 \subsection{Field theory in dimension 2} \label{sec:field}
In this paper, we are concerned with character stacks of closed oriented surfaces $\Sigma_g$ of genus $g$. Choosing a suitable finite set $A \subset \Sigma_g$ of order $g + 1$, we can decompose $\Sigma_g$ as 
\begin{equation}\label{eq:decomposition_of_Sigma_g}
    (\Sigma_g, A) = \bdpcupright \circ \left(\bdpgenus\right)^g \circ \bdpcupleft
\end{equation}
where the bordisms are given by
 \begin{equation}
     \label{eq:generators_bdp_2_Lambda}
     \begin{tikzcd}[row sep=0em] \bdpcupright & \bdpgenus & \bdpcupleft \\ D^\dag : (S^1, \point) \to \varnothing & L : (S^1, \point) \to (S^1, \point)  & D : \varnothing \to (S^1, \point) . \end{tikzcd}
 \end{equation}



Let us compute the field theory for the above bordisms. The fundamental groups $\pi_1\left(\bdpcupleftsmall\right)$ and $\pi_1\left(\bdpcuprightsmall\right)$ are trivial, implying $\mathfrak{X}_G\left(\bdpcupleftsmall\right) = \mathfrak{X}_G\left(\bdpcuprightsmall\right) = [\point/G]=\BG$. Since $\pi_1(S^1, \point) = \ZZ$, we have $\mathfrak{X}_G(S^1, \point) = [\Hom(\ZZ, G)/G] = [G/G]$ with the conjugation action. Therefore, the field theories on $\bdpcupleftsmall$ and $\bdpcuprightsmall$ are given by
\[ \mathcal{F}\left(\bdpcupleftsmall\right) = \left( \BG \xleftarrow{\id} \BG \rightarrow [G/G] \right) \quad \text{and} \quad  \mathcal{F}\left(\bdpcuprightsmall\right) = \left( [G/G] \leftarrow \BG \xrightarrow{\id} \BG \right) , \]
where the map $e:\BG\to [G/G]$ is induced from the map $\point\to G$ sending the point to the identity element of $G$. In particular, $Z\left(\bdpcupleftsmall\right): \K(\RStck/\BG)\to \K(\RStck/[G/G])$ is the map that sends an element $\mathfrak{X}\to \BG\in \RStck/\BG$ to the element $\mathfrak{X}\to \BG\to [G/G]\in \RStck/[G/G]$ using the map $e:\BG\to [G/G]$ as above. Similarly, $Z\left(\bdpcuprightsmall\right):\K(\RStck/[G/G])\to \K(\RStck/\BG)$ is the map that sends an element $\mathfrak{X}\xrightarrow{f} [G/G]\in \RStck/[G/G]$ to the element $\mathfrak{X}_e\to \BG\in \RStck/\BG$ where the map $\mathfrak{X}_e\to \BG$ is the map that is the map of the left-handside of the Cartesian product
\[
    \xymatrix{
        \mathfrak{X}_e \ar[r]\ar[d] & \mathfrak{X} \ar[d]^f \\
        \BG \ar[r]^e& [G/G].
        }
\]

Now, we turn our attention to the bordism $\bdpgenussmall: (S^1, \star) \to (S^1, \star)$ with two basepoints, let us call them $a$ and $b$. The surface of this bordism is homotopic to a torus with two punctures, so its fundamental group (based on $a$) is the free group on three generators $F_3$. We pick the generators $\gamma, \gamma_1, \gamma_2$ as depicted in the following image, and a path $\alpha$ connecting $a$ and $b$.
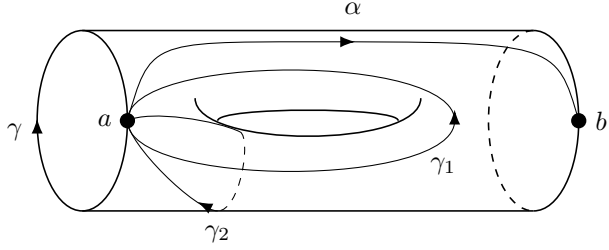
\[ \begin{tikzpicture}[semithick, scale=3.0]
    \begin{scope}
        \draw (-1,0) ellipse (0.2cm and 0.4cm);
        \draw (-1,0.4) -- (1,0.4);
        \draw (-1,-0.4) -- (1,-0.4);
        \draw (1,0.4) arc (90:-90:0.2cm and 0.4cm);
        \draw[dashed] (1,0.4) arc (90:270:0.2cm and 0.4cm);
        
        \draw (-0.5,0.1) .. controls (-0.5,-0.125) and (0.5,-0.125) .. (0.5,0.1);
        \draw (-0.4,0.0) .. controls (-0.4,0.0625) and (0.4,0.0625) .. (0.4,0.0);
        
        \draw[black,fill=black] (-0.8,0) circle (.2ex);
        \draw[black,fill=black] (1.2,0) circle (.2ex);
        
        \draw[thin] (-0.8,0) .. controls (-0.7,0.35) .. (0,0.35) .. controls (1.1,0.35) .. (1.2,0);
        \draw[thin] (-0.8,0) .. controls (-0.8,0.3) and (0.65,0.3) .. (0.65,0) .. controls  (0.65,-0.3) and (-0.8,-0.3) .. (-0.8,0);
        \draw[thin] (-0.8,0) .. controls (-0.7,0.05) and (-0.4,0.0) .. (-0.3,-0.05);
        \draw[thin, dashed] (-0.3,-0.05) .. controls (-0.25,-0.08) and (-0.3,-0.4) .. (-0.4,-0.4);
        \draw[thin] (-0.4,-0.4) .. controls (-0.5,-0.4) and (-0.8,-0.1) .. (-0.8,0);
        
        \draw[-{Latex}] (-1.2,0) -- +(0,0.01);
        \draw[-{Latex}] (0.2,0.35) -- +(0.01,0);
        \draw[-{Latex}] (0.65,0) -- +(0,0.05);
        \draw[-{Latex}] (-0.47,-0.37) -- +(-0.02,0.01);
        
        \node at (-1.3,-0.05) {$\gamma$};
        \node at (0.6,-0.2) {$\gamma_1$};
        \node at (-0.4,-0.5) {$\gamma_2$};
        \node at (0.2,0.5) {$\alpha$};
        \node at (-0.9,0) {$a$};
        \node at (1.3,0) {$b$};
    \end{scope}
\end{tikzpicture} \]
Using the generators, we identify the representation variety corresponding to $\bdpgenussmall$ as
\begin{align*}
    R_G\left(\bdpgenussmall\right) &\simeq \Hom(F_3, G) \times G \simeq G^4 \\
    \rho &\mapsto (\rho(\gamma), \rho(\gamma_1), \rho(\gamma_2), \rho(\alpha)) 
\end{align*}
and thus the corresponding character stack is given by $\mathfrak{X}_G\left(\bdpgenussmall\right) = [G^4/G]$, where $G$ acts by simultaneous conjugation on $G^4$.
A generator for $\pi_1(S^1, b)$ is given by $\alpha \gamma [\gamma_1, \gamma_2] \alpha^{-1}$, and so the field theory for $\bdpgenussmall$ is found to be
\[ \mathcal{F}\left(\bdpgenussmall\right) = \left( [G/G] \xleftarrow{\bar{p}} [G^4/G] \xrightarrow{\bar{q}} [G/G] \right) \]
induced by the morphisms
\begin{equation}
    \label{eq:span_Z_L}
    \begin{tikzcd}[row sep=0em]
        G & \arrow[swap]{l}{p} G^4 \arrow{r}{q} & G \\
        g & \arrow[mapsto]{l} (g, g_1, g_2, h) \arrow[mapsto]{r} & h g [g_1, g_2] h^{-1} .
    \end{tikzcd}
\end{equation}

After discussing the field theories corresponding to simple bordisms, we are ready to express the class of the character stack $\mathfrak{X}_G(\Sigma_g, \point)$ in terms of the TQFT. First of all, using \eqref{eq:the_hom_equation2}, we have that
\[\mathfrak{X}_G(\Sigma_g,\point)\times{_\BG} [G^{|A|-1}/G]=\mathfrak{X}_G(\Sigma_g,A).\]
Therefore, applying \eqref{eq:decomposition_of_Sigma_g} with $|A|=g+1$, we have
\begin{equation}
\label{eq:representation_variety_from_TQFT}
    [ \mathfrak{X}_G(\Sigma_g, \point) ] \cdot [G/G]^{g} = Z\left(\bdpcuprightsmall\right) \circ Z\left(\bdpgenussmall\right)^g \circ Z\left(\bdpcupleftsmall\right)(1),  
\end{equation}
where the multiplication on the left-hand side is given by the multiplication on the ring $\K(\Stck/\BG)$.

\begin{remark}\label{rem:special}
Note that the class $[G/G]$ might be a zero-divisor. In this paper, we focus on the affine linear group $G=\AGL_1(\CC)$ that contains the class of the affine line as a factor in their class in $\K(\Var_{k})$, which is a zero divisor (\cite{Borisov2014, Martin2016}) in $\K(\Var_k)$. 

We have two choices in computing the class $[ \mathfrak{X}_G(\Sigma_g, \point) ]$. Either, we consider the localization of $\K(\Stck/\BG)$ with the class of $[G/G]$, in which the computation 
\begin{equation}
    \label{eq:character_stack_from_TQFT}
     [\mathfrak{X}_G(\Sigma_g, \point) ]= \frac{1}{[G/G]^{g}} Z\left(\bdpcuprightsmall\right) \circ Z\left(\bdpgenussmall\right)^g \circ Z\left(\bdpcupleftsmall\right)(1) 
\end{equation}
holds. Or, we consider the evaluation map \eqref{eq:evaluation}
\[ \K(\Stck/BG)\to \hat{\K}(\Var_{k})\]
in which 
\begin{equation}\label{eq:representation_variety_from_TQFT2}
[\mathfrak{X}_G(\Sigma_g, \point) ]= \frac{1}{[G]^{g-1}}Z\left(\bdpcuprightsmall\right) \circ Z\left(\bdpgenussmall\right)^g \circ Z\left(\bdpcupleftsmall\right)(1) \end{equation}
holds for any special algebraic group $G$. The latter approach has a crucial short-coming, it forgets the group action of $G$ on the representation variety. On the other hand, we will use \eqref{eq:representation_variety_from_TQFT2} to compare the virtual class of the character stack in $\K(\Stck/\BG)$ with the virtual class of the representation variety in $\hat{\K}(\Var_{k})$ in the case of $\AGL_1(\CC)$ (\cite{gonlogmun20, habvog20}).
\end{remark}

\subsection{Simplification of the TQFT}
\label{sec:simplification_TQFT}

Recall that the morphism $Z(\bdpgenussmall)$ is given by $q_! \circ p^*$, where $p$ and $q$ are given by the span
\[ 
\begin{tikzcd}[row sep=0em]
    {[G/G]} & {[G^4/G]} \arrow[swap]{l}{p} \arrow{r}{q} & {[G/G]} \\ g & \arrow[mapsto]{l} (g, g_1, g_2, h) \arrow[mapsto]{r} & h g [g_1, g_2] h^{-1} .
\end{tikzcd}
\]
The aim of this section is to show that, instead of these maps that involve an awkward conjugation, we can consider instead the `more practical maps' $\tilde{p}, \tilde{q}: [G^3/G] \to [G/G]$ as given by the span
\begin{equation}\label{eq:span-simplified}
     \begin{tikzcd}[row sep=0em]
    {[G/G]} & {[G^3/G]} \arrow[swap]{l}{\tilde{p}} \arrow{r}{\tilde{q}} & {[G/G]} \\ g & \arrow[mapsto]{l} (g, g_1, g_2) \arrow[mapsto]{r} & g [g_1, g_2] .
\end{tikzcd}
\end{equation}

Let us denote by $\tilde{\Theta}$ the quantization of the span (\ref{eq:span-simplified}), that is $\tilde{\Theta} = \tilde{q}_!\tilde{p}^*$. The following results show that $\tilde{\Theta}$ can be used instead of $Z\left(\bdpgenussmall\right)$ to compute the TQFT.



\begin{proposition}
    \label{prop:simplification_TQFT}
    For all $g \ge 0$ we have
    \[ Z \left( \bdpcuprightsmall \right) \circ Z\left(\bdpgenussmall\right)^g \circ Z\left( \bdpcupleftsmall \right) = [G/G]^g \cdot Z \left( \bdpcuprightsmall \right) \circ \tilde{\Theta}^g \circ Z\left( \bdpcupleftsmall \right) . \]
\end{proposition}

\begin{proof}
    We prove the more general statement that $Z \left( \bdpcuprightsmall \right) \circ Z\left(\bdpgenussmall\right)^g = [G/G]^g \cdot Z \left( \bdpcuprightsmall \right) \circ \tilde{\Theta}^g$, by induction on $g$, where the case $g = 0$ is trivial. Suppose the statement holds for some $g \ge 0$, and consider any element $[X/G] \xrightarrow{f} [G/G] \in \RStck/[G/G]$. Then, a direct computation shows the following equivalences.
    \small
    \begin{align*}
        Z & \left( \bdpcuprightsmall \right)  \circ Z\left(\bdpgenussmall\right)^{g + 1}([X/G]) = [G/G]^g \cdot Z \left( \bdpcuprightsmall \right) \circ \tilde{\Theta}^g \circ Z\left(\bdpgenussmall\right)\left([X/G]\right) \\
            &= [G/G]^g \cdot \left[ \left\{ (x, A, B, h, A_1, B_1, \ldots, A_g, B_g) \in X \times G^{3 + 2g} \;\;\left|\;\; h f(x) [A, B] h^{-1} \prod_{i = 1}^{g} [A_i, B_i] = 1 \right.\right\} / G \right] \\
            &= [G/G]^g \cdot \left[ \left\{ (x, A, B, h, A'_1, B'_1, \ldots, A'_g, B'_g) \in X \times G^{3 + 2g} \;\;\left|\;\; h f(x) [A, B] \prod_{i = 1}^{g} [A'_i, B'_i] h^{-1} = 1 \right. \right\} / G \right] \\
            &= [G/G]^g \cdot \left[ \left\{ (x, A, B, h, A'_1, B'_1, \ldots, A'_g, B'_g) \in X \times G^{3 + 2g} \;\;\left|\;\; f(x) [A, B] \prod_{i = 1}^{g} [A'_i, B'_i] = 1 \right\} / G \right. \right] \\
            &= [G/G]^{g + 1} \cdot Z \left( \bdpcuprightsmall \right) \circ \tilde{\Theta}^{g + 1}([X/G]) ,
    \end{align*}\normalsize
    where in the second equality we used the induction hypothesis and the third one is obtained by using the substitutions $A'_i = h^{-1} A_i h$ and $B'_i = h^{-1} B_i h$. Hence the statement also holds for $g + 1$.
\end{proof}

In fact, by computing the map $\tilde{\Theta}^g$ explicitly, we get an even stronger result.

\begin{corollary}
    \label{cor:simplification-TQFT}
    For any $g \geq 0$, the class of the character stack is given by
    \begin{equation}
        \label{eq:character-stack-simplified}
        [\mathfrak{X}_G(\Sigma_g, \point)] =Z \left( \bdpcuprightsmall \right) \circ \tilde{\Theta}^g \circ Z \left( \bdpcupleftsmall \right) (1) .
    \end{equation}
\end{corollary}

\begin{proof}
    A similar computation as in the proof of Proposition \ref{prop:simplification_TQFT} shows the following
     \[Z \left( \bdpcuprightsmall \right) \circ \tilde{\Theta}^g \circ Z \left( \bdpcupleftsmall \right) (1)=Z \left( \bdpcuprightsmall \right) \circ \tilde{\Theta}^g(BG\rightarrow [G/G])=\]
    \[=\left[ \left\{ (A_1, B_1, A_2, B_2, \ldots, A_g, B_g) \in G^{2g} \;\;\left|\;\; \prod_{i = 1}^{g} [A_i, B_i]= 1 \right. \right\} / G \right]= [\mathfrak{X}_G(\Sigma_g, \point)].\]
\end{proof}

\begin{remark}
    This is an improvement with respect to (\ref{eq:representation_variety_from_TQFT}), since there are no extra factors $[G/G]^g$ to remove.
\end{remark}

Hence, in order to compute the virtual class of the character stack, it suffices to do the computations with the practical map $\tilde{\Theta}$. 
For this reason, in the upcoming sections, we will write
\[ Z'\left( \bdpgenussmall \right) = \tilde{\Theta}, \]
to keep the connection with the bordisms, even though $Z'$ is not a functor, nor a TQFT on $\Bdp_2$.

\begin{remark}
There is a slightly more abstract way of understanding the previous computation. Given a span $S: \mathfrak{X}_1 \stackrel{f}{\leftarrow} \mathfrak{Z} \stackrel{g}{\rightarrow} \mathfrak{X}_2$ of $\mathfrak{S}$-stacks and a morphism $h: \mathfrak{X}_2 \times \mathfrak{Y} \to \mathfrak{X}_2$, let us denote by $h \star S$ the span 
\[ \begin{tikzcd}[column sep=4em]
    {\mathfrak{X}_1} & {\mathfrak{Z} \times_{\mathfrak{S}} \mathfrak{Y}} \arrow[swap]{l}{f \circ \pi_{\mathfrak{Z}}} \arrow{r}{h \circ \left(g \times \id\right)} & {\mathfrak{X}_2}.
\end{tikzcd} \]

With this notion, Proposition \ref{prop:simplification_TQFT} actually shows that if we consider the map $c_g: [G/G] \times_{\BG} [G^g/G] \to [G/G]$ given by $(w, h_1, \ldots, h_g) \mapsto h_g\cdots h_1 w h_1^{-1} \cdots h_g^{-1}$, then for any $g \geq 1$ we have
\[
    \mathcal{F}\left(\bdpgenussmall\right)^g = c_g \star \mathcal{F}'\left(\bdpgenussmall\right)^g,
\]
where again $\mathcal{F}'\left(\bdpgenussmall\right)^g$ is an abuse of notation to denote the span (\ref{eq:span-simplified}). In this setting, capping with the bordism $\bdpcuprightsmall$ removes the effect of $c_g \star$ and turns it into a simple factor $[G/G]^g$.
\end{remark}

\begin{remark}\label{rmk:conjugation-is-needed}
    Note that this technique of removing conjugations only works for one hole at a time. Although tempting, it is not possible to define a TQFT without such conjugations. If we were to define a field theory $\mathcal{F}'$ with
    \[ \mathcal{F}' \left(\bdppantsrightsmall\right) = \left(\begin{tikzcd}[row sep=0em]
            {[G^2/G]} & {[G^2/G]} \arrow[swap]{l}{} \arrow{r}{} & {[G/G]} \\ (g_1, g_2) & \arrow[mapsto]{l} (g_1, g_2) \arrow[mapsto]{r} & g_1 g_2
        \end{tikzcd}\right), \]
    \[ \mathcal{F}' \left(\bdppantsleftsmall\right) = \left(\begin{tikzcd}[row sep=0em]
            {[G/G]} & {[G^2/G]} \arrow[swap]{l}{} \arrow{r}{} & {[G^2/G]} \\ g_1 g_2 & \arrow[mapsto]{l} (g_1, g_2) \arrow[mapsto]{r} & (g_1, g_2) .
        \end{tikzcd}\right) , \]
    then, a direct computation would show that
    $$\mathcal{F}' \left(\bdppantsrightsmall\right) \circ \mathcal{F}' \left(\bdppantsleftsmall\right) = \left(\begin{tikzcd}[row sep=0em]
            {[G/G]} & {[G^2/G]} \arrow[swap]{l}{} \arrow{r}{} & {[G/G]} \\ g_1 g_2 & \arrow[mapsto]{l} (g_1, g_2) \arrow[mapsto]{r} & g_1 g_2
        \end{tikzcd}\right) \ne \mathcal{F}' \left( \bdpgenussmall \right),
    $$
    implying that $\mathcal{F}'$ cannot be a functor.
        
    This calculation evinces that the path appearing in the fundamental groupoid of  $\bdppantsrightsmall$ joining the two components of the out-boundary is crucial here: it is half of one of the loop generators of the fundamental group of $\bdpgenussmall$. This point is related to the fact that the classical Seifert-van Kampen theorem for fundamental groups only works if the intersection of the open sets considered is path connected. Otherwise said, it is mandatory to use fundamental groupoids with at least one basepoint at each component.
\end{remark}


In relation to the previous remark, the reason why the TQFT can be simplified in the 2-dimensional case is the following. In the 2-dimensional case, there exists a natural embedding of categories
$$
    \Tb_{2} \hookrightarrow \Bdp_{2}
$$
where, $\Tb_2$ is the so-called \emph{category of $2$-dimensional tubes without basepoints}, which is the wide subcategory of $\Bdp_2$ \emph{with basepoints} with morphisms whose connected components have only connected (maybe empty) in and out boundaries. The embedding is given by assigning each compact 1-dimensional manifold $X$ to the tuple $(X, A)$, where $A$ is a collection of basepoints with a single point per connected component. For a tube $\Sigma$, we assign the tuple $(\Sigma, A)$ where $A$ is a set of $-(\chi(\Sigma) + b^0(\partial \Sigma) -4)/2$ basepoints, with one on each connected component of the boundary $\partial \Sigma$ and the remaining in the interior of $\Sigma$ (see also \cite{arXiv181009714}). 

Thanks to this embedding, the lax monoidal TQFT
$
    Z: \Bdp_2 \to \K(\RStck/\BG)\text{-}\Mod
$
descends to a lax monoidal functor
$$
    Z': \Tb_2 \to \K(\RStck/\BG)\text{-}\Mod.
$$
Notice that subcategory $\Tb_2$ excludes for instance the pair of pants $\bdpantsrightsmall$, and thus the obstruction of Remark \ref{rmk:conjugation-is-needed} does not arise in $\Tb_2$.

The existence of such embedding $\Tb_{2} \hookrightarrow \Bdp_{2}$ of basepoint-free tubes is a feature occurring only in dimension $2$. In this direction, the shortcut developed in this section does not generalize to $\Bdp_n$ for $n \geq 3$. Keeping track of the basepoints in the bordism is thus compulsory in general and an essential feature of the TQFT provided that we want to compute virtual classes of character stacks as $\BG$-stacks, or equivalently, understanding the $G$-equivariant theory of the representation variety. There actually exists a parallel TQFT without basepoints for character stacks, as developed in \cite{gonzalez2023arithmetic}, but it can only compute their virtual class as regular stacks and thus the equivariant information is lost.

\section[AGL1(k)-character stacks]{$\AGL_1(k)$-character stacks}\label{sec:AGL1-character-stack}

In this section, we compute the virtual classes of character stacks of surface groups corresponding to the affine algebraic group
\[ G =\AGL_1(k) = \left\{ \begin{pmatrix} a & b \\ 0 & 1 \end{pmatrix} : a\ne 0\right\} , \]
where $k$ is any field. Notice that the class of $\AGL_1(k)$ in $\K(\Var_{k})$ is $q(q-1)$ where $q = [\AA^1_k]$ is the class of the affine line.
We shall use the following stratification of $G$:
\[ I = \left\{ \begin{pmatrix} 1 & 0 \\ 0 & 1 \end{pmatrix} \right\}, \qquad J = \left\{ \begin{pmatrix} 1 & b \\ 0 & 1 \end{pmatrix} : b \ne 0 \right\}, \qquad M = \left\{ \begin{pmatrix} a & b \\ 0 & 1 \end{pmatrix} : a \ne 0, 1 \right\} , \]
which induces the decomposition
\begin{equation}
    \label{eq:decomposition_grothendieck_AGL1}
    \K(\RStck/[G/G]) = \K(\RStck/[I/G]) \oplus \K(\RStck/[J/G]) \oplus \K(\RStck/[M/G]) .
\end{equation}
We denote the unit elements of the rings $\K(\RStck/[I/G])$ and $\K(\RStck/[J/G])$ respectively by
\begin{align*}
    \mathbf{1}_I &= \Big([I/G] \to [I/G]\Big) \in \K(\RStck/[I/G]), \\
    \mathbf{1}_J &= \Big([J/G] \to [J/G]\Big) \in \K(\RStck/[J/G]).
\end{align*}
%

Recall that the natural map $[G/G]\to \BG$ induces a $\K(\RStck/\BG)$-module structure on $\K(\RStck/[G/G])$ (see Remark \ref{rem:modulestructure}). As the computations in Propositions \ref{prop:agl1} and \ref{prop:agl2} will show, the $\K(\RStck/\BG)$-submodule of $\K(\RStck/[G/G])$ generated by $\{ \mathbf{1}_I, \mathbf{1}_J \}$ is invariant under $Z(\bdpgenussmall)$ and $Z'(\bdpgenussmall)$.
%
%
As a result, in order to compute the character stack $\mathfrak{X}_{\AGL_1(k)}(\Sigma_g, \point)$ of surface groups, it is enough to compute $Z'\left(\bdpgenussmall\right)$ on the submodule generated by the basis $\{\mathbf{1}_I, \mathbf{1}_J \}$. Indeed, $ Z \left( \bdpcupleftsmall \right) (1)$ is simply $\mathbf{1}_I$, furthermore $Z\left(\bdpcuprightsmall\right)$ sends $\mathbf{1}_I$ to 1 and $\mathbf{1}_J$ to 0.

Let us begin with some algebraic relations in $\K(\RStck/\BG)$. There are two special elements to consider. First, the group $G$ acts naturally on the affine line by scaling and translation
\[ \begin{pmatrix} a & b \\ 0 & 1 \end{pmatrix} \cdot x = ax + b \text{ for } x \in \GG_a, \]
and we denote the corresponding the quotient stack by $[\GG_a/G]$.
Also $G$ acts naturally on the punctured affine line by scaling
\[ \begin{pmatrix} a & b \\ 0 & 1 \end{pmatrix} \cdot x = ax \text{ for } x \in \GG_m , \]
and we denote the corresponding quotient stack by $[\GG_m/G]$.

\begin{lemma}\label{lem:relationsinBG}
    In $\K(\RStck/\BG)$ we have the following relations:
    \begin{align*}
        \label{eq:relations_Gm_Ga}
        [\GG_m/G]^2 &= (q - 1) [\GG_m/G], \\
        [\GG_a/G]^2 &= [\GG_a/G] + [\GG_a/G] [\GG_m/G], \\
        [\AGL_1(k)/G] &= [\GG_a / G][\GG_m / G].
    \end{align*}
Here, $[\AGL_1(k)/G]$ denotes the transitive action of $G$ on itself by multiplication on the left.
\end{lemma}

\begin{proof}
    For the first relation, consider the isomorphism
    \[\GG_m \times \GG_m \to \GG_m \times \GG_m, \quad (x, y) \mapsto (x/y, y) . \]
    The scaling action on $x$ and $y$ yields a trivial action on $x/y$, so the statement follows.
    For the second relation, consider the piece-wise isomorphism
    \[\GG_a \times \GG_a \to \left(\{ 0 \} \times \GG_a\right) \sqcup \left(\GG_m \times \GG_a\right), \quad (x, y) \mapsto (x - y, y) . \]
    Indeed, cutting $\GG_a \times \GG_a$ in two pieces, the diagonal subvariety $\{ (x, x) \}$ is isomorphic to $\GG_a$, and the open complement $\{ (x, y) : x \ne y \}$ is mapped to $\GG_m \times \GG_a$. This isomorphism is equivariant implying the statement. Finally, the third relation follows from the fact that for $\smatrix{a & b \\ 0 & 1} \in G$, the coordinate $a$ transforms like $\GG_m$, and the coordinate $b$ like $\GG_a$.
\end{proof}


\begin{lemma}
    Under the natural map $[G/G]\to \BG$, we have
    \[ [I/G] = \BG, \qquad [J/G] = [\GG_m / G], \qquad [M/G] = (q - 2)[\GG_a / G] . \]
    in $\K(\RStck/\BG)$. In particular, 
    $[G/G] = \BG + [\GG_m / G] + (q - 2)[\GG_a / G]$
    in $\K(\RStck/\BG)$.
\end{lemma}

\begin{proof}
    Consider two matrices $A = \begin{pmatrix} a & b \\ 0 & 1 \end{pmatrix}$ and $B = \begin{pmatrix} x & y \\ 0 & 1 \end{pmatrix}$. Then, 
    \[ABA^{-1}=\begin{pmatrix} x & b(1-x)+ay\\ 0 & 1\end{pmatrix}
    .\]
    This shows that if $B\in J$, then $ABA^{-1} = \begin{pmatrix} 1 & ay \\ 0 & 1\end{pmatrix}$, so $A$ acts by conjugation on $J$ as scaling on $\GG_m$. Similarly, if $x \ne 1$, that is $B\in M$, then we have an isomorphism
    \[M \to (\AA^1_k \setminus \{ 0, 1 \}) \times \GG_a, \quad \begin{pmatrix} x & y \\ 0 & 1 \end{pmatrix} \mapsto (x, y / (1 - x)), \]
    which is $\AGL_1(k)$-equivariant (here $(\AA^1_k\setminus \{0,1\})$ is endowed with the trivial action). The rest of the statement is immediate.
\end{proof}

Now, we are ready to compute the TQFT. We will compute $Z'(\bdpgenussmall)$ on the submodule generated by the basis $\{\mathbf{1}_I, \mathbf{1}_J\}$, starting with the image of $\mathbf{1}_I$.

\begin{proposition}
    \label{prop:agl1}
    Under the decomposition \eqref{eq:decomposition_grothendieck_AGL1} we have that
    \[ Z'(\bdpgenussmall)(\mathbf{1}_I)= \underbrace{\left(1 + (q + 1)[\GG_m / G] + q(q - 2) [\GG_a / G]\right) \mathbf{1}_I}_{\in \K(\RStck/[I/G])} + \underbrace{q(q - 2) [\GG_a/G] \mathbf{1}_J}_{\in \K(\RStck/[J/G])}.\]
    (Here the multiplication is given by the $\K(\RStck/\BG)$-module structures on $\K(\RStck/[I/G])$ and $\K(\RStck/[J/G])$ induced by the natural maps $[I/G]\to \BG$ and $[J/G]\to \BG$.)
\end{proposition}

\begin{proof}
First we compute $Z'(\bdpgenussmall)(\mathbf{1}_I)$ restricted to $[I/G]$.
This gives the class of the substack $[\{ (g_1,g_2) : [g_1,g_2] = 1\} / G]$ of $[G^2/G]$ in $\K(\RStck/[I/G])$. We stratify the substack into three pieces:
\begin{itemize}
    \item Case $g_1 = \id$. In this case $g_2$ can be anything, so we obtain the class $[G/G] \mathbf{1}_I= (1 + [\GG_m / G] + (q - 2)[\GG_a / G])\mathbf{1}_I$.
    \item Case $g_1 \in J$. In this case $g_2$ has to be either $\id$ or an element of $J$. We have a $G$-equivariant isomorphism
    \[\AA^1_k \times J\to \{(g_1,g_2):[g_1,g_2]=1, g_1\in J\}, \quad \left(t,\begin{pmatrix} 1 & y\\ 0& 1\end{pmatrix}\right)\mapsto \left(\begin{pmatrix} 1 & y\\ 0& 1\end{pmatrix}, \begin{pmatrix} 1 & ty \\ 0& 1\end{pmatrix}\right) , \]
    so we obtain the class $q[J/G]\mathbf{1}_I=q[\GG_m/G]\mathbf{1}_I$.
    \item Case $g_1 \in M$. In this case, either $g_2=\id$ or $g_2\in M$. We have a $G$-equivariant isomorphism
    \[\left(\AA^1_k\setminus \{0\}\right)\times M \to \{(g_1,g_2):[g_1,g_2]=1, g_1\in M\}, \quad \left(t,\begin{pmatrix} x & y\\ 0& 1\end{pmatrix}\right)\mapsto \left(\begin{pmatrix} x & y\\ 0& 1\end{pmatrix}, \begin{pmatrix} tx & \frac{y(1-tx)}{1-x}\\ 0& 1\end{pmatrix}\right) , \]
    so we obtain the class $(q-1)[M/G]\mathbf{1}_I=(q-1)(q-2)[\GG_a/G]\mathbf{1}_I$.

\end{itemize}
Summarizing the above discussion, we obtain the class
\[\left(1 + (q+1)[\GG_m / G] + q(q - 2)[\GG_a / G]\right) \mathbf{1}_I \in \K(\RStck/[I/G]).\]

Now, we restrict the element $Z'(\bdpgenussmall)(\mathbf{1}_I)$ to the stratum $[J/G]$. In this case, we obtain the class of the substack $[\{g_1,g_2:[g_1,g_2]\in J\}/G]$ of $[G^2,G]$ in $\K(\RStck/[J/G])$. As before, we stratify the substack into three pieces.
\begin{itemize}
    \item Case $g_1 = \id$. There are no solutions.
    \item Case $g_1 \in J$. In this case $g_2$ has to be in $M$. Explicitly, we have a $G$-equivariant isomorphism
    \[\left(\AA^1_k\setminus \{0,1\}\right)\times \GG_a\times J\to \{(g_1,g_2):[g_1,g_2]\in J, g_1\in J\}\]
    given by
    \[\left(x,t,\begin{pmatrix} 1 & b\\ 0& 1\end{pmatrix}\right)\mapsto \left(\begin{pmatrix} 1 & \frac{b}{1-x}\\ 0& 1\end{pmatrix}, \begin{pmatrix} x & t(1-x)\\ 0& 1\end{pmatrix}\right).\]
    Composing this isomorphism with the commutator map
    \[\{(g_1,g_2):[g_1,g_2]\in J, g_1\in J\}\mapsto J, \qquad (g_1,g_2)\mapsto [g_1,g_2],\] 
    we obtain the trivial fibration
    \[\left(\AA^1_k\setminus \{0,1\}\right)\times \GG_a\times J\to J\]
    given as projection onto the third component. This provides the class $(q-2)[\GG_a/G]\mathbf{1}_J$ in $\K(\RStck/[J/G])$.
    \item Case $g_1 \in M$. We have a $G$-equivariant isomorphism
    \[\left(\AA^1_k\setminus \{0,1\}\right)\times \left(\AA^1_k\setminus \{0\}\right)\times \GG_a\times J\to \{(g_1,g_2):[g_1,g_2]\in J, g_1\in M\}\]
    given by
    \[\left(a,x,t,\begin{pmatrix} 1 & b\\ 0& 1\end{pmatrix}\right)\mapsto \left(\begin{pmatrix} a & t(1-a)\\ 0& 1\end{pmatrix}, \begin{pmatrix} x & \frac{b-t(1-a)(1-x)}{a-1}\\ 0& 1\end{pmatrix}\right).\]
    Again, under this isomorphism the map onto $J$ becomes the projection onto the third component, so we get the class $(q-2)(q-1)[\GG_a/G]\mathbf{1}_J$ in $\K(\RStck/[J/G])$.
\end{itemize}
In total, we obtain the class $q(q-2)[\GG_a/G]\mathbf{1}_J$ in $\K(\RStck/[J/G])$. This concludes the proof.
\end{proof}

Next, we compute the image of $\mathbf{1}_J$ under $Z'(\bdpgenussmall)$.

\begin{proposition}
    \label{prop:agl2}
    Under the decomposition \eqref{eq:decomposition_grothendieck_AGL1}  we have that
    \[ Z'(\bdpgenussmall)(\mathbf{1}_J) = \left(q(q - 2)[\AGL_1(k)/G]\right)\mathbf{1}_I + \left((q^2 + q(q - 1)(q - 2)[\GG_a/G])\right)\mathbf{1}_J. \]
\end{proposition}

\begin{proof}
    First, we compute $Z'(\bdpgenussmall)(\mathbf{1}_J)$ restricted to $[I/G]$. Since $g [g_1, g_2] = 1$ is equivalent to $[g_1,g_2] = g^{-1}$, the class $Z'(\bdpgenussmall)(\mathbf{1}_J)$ over $[I/G]$ is the same as the class of $Z'(\bdpgenussmall)(\mathbf{1}_I)$ over $[J/G]$ regarded as a class in the ring $\K(\RStck/[I/G]) \simeq \K(\RStck/\BG)$. Thus, we obtain the class
    \[ q(q-2)[\GG_a/G][J/G]\mathbf{1}_I = q(q-2)[\AGL_1(k)/G]\mathbf{1}_I \]
    in $\K(\RStck/[I/G])$.
    
    Now, we compute $Z'(\bdpgenussmall)(\mathbf{1}_J)$ restricted to $[J/G]$. We stratify the stack 
    \[ [\{ (g, g_1, g_2) : g \in J, g[g_1, g_2] \in J \} / G]\]
    as follows.
    \begin{itemize}
        \item Case $\tr g_1=\tr g_2=2$. In this case, $[g_1,g_2]=\id$. Thus, $g[g_1,g_2]=g\in J$, providing the class $q^2\mathbf{1}_J$ in $\K(\RStck/[J/G])$.
        \item Case $\tr g_1\ne 2$. We have a $G$-equivariant isomorphism
        \[\left(\AA^1_k\setminus \{0\}\right)\times \left(\AA^1_k\setminus \{0\}\right)\times \left(\AA^1_k\setminus \{0,1\}\right)\times \GG_a\times J\to \{(g,g_1,g_2):g\in J, g[g_1,g_2]\in J, g_1\in M\}\]
        given by
        \[\left(u, x,a,t,\begin{pmatrix} 1 & b\\ 0& 1\end{pmatrix}\right)\mapsto \left(\begin{pmatrix} 1 & ub\\ 0& 1\end{pmatrix}, \begin{pmatrix} a & t(1-a)\\ 0& 1\end{pmatrix}, \begin{pmatrix} x & \frac{b-ub-t(1-a)(1-x)}{a-1}\\ 0& 1\end{pmatrix}\right).\]
        Composing with the map 
        \[\{(g,g_1,g_2):g\in J, g[g_1,g_2]\in J, g_1\in M\}\to J\]
        sending $(g,g_1,g_2)\mapsto g[g_1,g_2]$, we obtain a trivial fibration
        \[\left(\AA^1_k\setminus \{0\}\right)\times \left(\AA^1_k\setminus \{0\}\right)\times \left(\AA^1_k\setminus \{0,1\}\right)\times \GG_a\times J\to J\]
        providing the class $(q-1)^2(q-2)[\GG_a/G]\mathbf{1}_J$ in $\K(\RStck/[J/G])$.
        \item Case $\tr g_1 = 2$ and $\tr g_2\ne 2$. We have a $G$-equivariant isomorphism
        \[\left(\AA^1_k\setminus \{0\}\right)\times \left(\AA^1_k\setminus \{0,1\}\right)\times \GG_a\times J\to \{(g,g_1,g_2):g\in J, g[g_1,g_2]\in J, g_1\not \in M, g_2\in M\}\]
        given by
        \[\left(u, x, t, \begin{pmatrix} 1 & b\\ 0& 1\end{pmatrix}\right)\mapsto \left(\begin{pmatrix} 1 & ub\\ 0& 1\end{pmatrix}, \begin{pmatrix} 1 & \frac{t(1-u)}{1-x}\\ 0& 1\end{pmatrix}, \begin{pmatrix} x & t(1-x)\\ 0& 1\end{pmatrix}\right).\]
        As before, we obtain a fibration
        \[\left(\AA^1_k\setminus \{0\}\right)\times \left(\AA^1_k\setminus \{0,1\}\right)\times \GG_a\times J\to J\]
        providing the class $(q-1)(q-2)[\GG_a/G]\mathbf{1}_J$ in $\K(\RStck/[J/G])$.
    \end{itemize}
    In total, we obtain that 
    \[ Z'(\bdpgenussmall)(\mathbf{1}_J)=\left(q(q - 2)[\GG_a/G][\GG_m/G]\right)\mathbf{1}_I+\left(q^2 + q(q - 1)(q - 2)[\GG_a/G]\right)\mathbf{1}_J. \]
\end{proof}

Putting together Propositions \ref{prop:agl1} and \ref{prop:agl2}, we obtain the following description of the TQFT.

\begin{theorem}
    \label{thm:AGL1-matrix}
    The $\K(\RStck/\BG)$-submodule generated by $\{ \mathbf{1}_I, \mathbf{1}_J \}$ is invariant under $Z(\bdpgenussmall)$ and $Z'(\bdpgenussmall)$. Explicitly, with respect to the basis $\{\mathbf{1}_I, \mathbf{1}_J \}$, we have
    \[ Z'(\bdpgenussmall) = \left[\begin{matrix}
        1+\LL \left(\LL - 2\right)[\GG_a/G]  + \left(\LL + 1\right)[\GG_m/G]  & \LL \left(\LL - 2\right)[\AGL_1(k)/G] \\
        \LL \left(\LL - 2\right)[\GG_a/G] & \LL^{2} + \LL \left(\LL-1\right)\left(\LL - 2\right)[\GG_a/G]
    \end{matrix}\right] . \]
\end{theorem}

This description enables us to compute the virtual classes of character stacks $[\mathfrak{X}_{\AGL_1(k)}(\Sigma_g, \point)]\in \K(\RStck/\BG)$ using equation \eqref{eq:character-stack-simplified}, 
\[ [\mathfrak{X}_{\AGL_1(k)}(\Sigma_g, \point)] = Z\left(\bdpcuprightsmall \right) \circ Z' \left(\bdpgenussmall \right)^g \circ Z\left(\bdpcupleftsmall\right) (1) . \]
Specifically, the class $[\mathfrak{X}_{\AGL_1(k)}(\Sigma_g, \point)]$ is the top left entry of the $g$-th power of the matrix of $Z'(\bdpgenussmall)$. While taking powers of the matrix, by Lemma \ref{lem:relationsinBG}, new virtual classes appear. Thus, in order to compute the powers of the matrix, we expand the matrix of $Z'(\bdpgenussmall)$ to a larger matrix using the classes 
\[\{\mathbf{1}_I, [\GG_a/G]\mathbf{1}_I, [\GG_m/G]\mathbf{1}_I, [\AGL_1(k)/G] \mathbf{1}_I \}\]
in $\K(\RStck/[I/G])$, and the classes 
\[\{\mathbf{1}_J, [\GG_a/G]\mathbf{1}_J, [\GG_m,G]\mathbf{1}_J, [\AGL_1(k)/G] \mathbf{1}_J\}\] in $\K(\RStck/[J/G])$.
In terms of this new generating set, the matrix of $Z'(\bdpgenussmall)$ is expressed as
\[ \scalebox{0.85}{$\left[\begin{matrix}1 & 0 & 0 & 0 & 0 & 0 & 0 & 0\\\LL \left(\LL - 2\right) & \left(\LL - 1\right)^{2} & 0 & 0 & 0 & 0 & 0 & 0\\\LL + 1 & 0 & \LL^{2} & 0 & 0 & 0 & 0 & 0\\0 & \LL^{2} - \LL + 1 & \LL \left(\LL - 2\right) & \LL^{2} \left(\LL - 1\right) & \LL \left(\LL - 2\right) & \LL^{2} \left(\LL - 2\right) & \LL \left(\LL - 2\right) \left(\LL - 1\right) & \LL^{2} \left(\LL - 2\right) \left(\LL - 1\right)\\0 & 0 & 0 & 0 & \LL^{2} & 0 & 0 & 0\\\LL \left(\LL - 2\right) & \LL \left(\LL - 2\right) & 0 & 0 & \LL \left(\LL - 2\right) \left(\LL - 1\right) & \LL \left(\LL^{2} - 2 \LL + 2\right) & 0 & 0\\0 & 0 & 0 & 0 & 0 & 0 & \LL^{2} & 0\\0 & \LL \left(\LL - 2\right) & \LL \left(\LL - 2\right) & \LL^{2} \left(\LL - 2\right) & 0 & \LL \left(\LL - 2\right) \left(\LL - 1\right) & \LL \left(\LL - 2\right) \left(\LL - 1\right) & \LL^{2} \left(\LL^{2} - 3 \LL + 3\right)\end{matrix}\right].$} \]
Diagonalizing $Z'(\bdpgenussmall)$ as $P D P^{-1}$, we obtain
\[ D = \textup{diag}(1, q^2, q^2, q^2, q^2, (q - 1)^2, q(q^2 - 2q + 2), q^2(q - 1)^2) \]
and
\[ P = \begin{bmatrix}
    q - 1&  0&     0&     0&          0&                                             0&     0& 0\\
 1 - q&  0&     0&     0&          0& (q - 1)(q^3 - 3q^2 + 4q - 1)&     0& 0\\
    -1& -q& 1 - q& 1 - q& -q(q - 1)&                                             0&     0& 0\\
     1&  1&     0&     0&          0&          -(q^3 - 3q^2 + 4q - 1)&     0& 1\\
     0&  0&    -1&     0&          0&                                             0&     0& 0\\
     0&  0&     1&     0&          0&                                         -q(q - 1)(q - 2) & 1 - q& 0\\
     0&  0&     0&     1&          0&                                             0&     0& 0\\
     0&  0&     0&     0&          1&                                             q(q - 2)&     1& 1\\
\end{bmatrix} \]

Now, the class $[\mathfrak{X}_{\AGL_1(k)}(\Sigma_g, \point)]$ is given by the first four entries of the first column of the matrix $P D^g P^{-1}$ (that is the virtual class in $\K(\RStck/[I/G])$), which are
\[ \renewcommand\arraystretch{1.25} \left[\begin{matrix}1\\\left(\LL - 1\right)^{2 g} - 1\\\frac{\LL^{2 g} - 1}{\LL - 1}\\\frac{\left(\LL^{2 g - 2} - 1\right) \left(\left(\LL - 1\right)^{2 g} - 1\right)}{\LL - 1}\end{matrix}\right] .\]

From this result, we obtain the virtual class of the character stack.

\begin{theorem}
    \label{thm:character_stack_AGL1}
    The virtual class of the character stack $[\mathfrak{X}_{\AGL_1(k)}(\Sigma_g)] \in \K(\RStck/\BG)$ equals
    \[ \BG + ((q - 1)^{2g} - 1)[\GG_a/G] + \frac{\LL^{2 g} - 1}{\LL - 1}[\GG_m/G] + \frac{\left(\LL^{2 g - 2} - 1\right) \left(\left(\LL - 1\right)^{2 g} - 1\right)}{\LL - 1}[\AGL_1(k)/G] . \] \qed
\end{theorem}

\begin{remark}
    The expression above makes sense even without localizing by $q-1$. Indeed, in the quotients $(\LL^{2 g} - 1)/(\LL - 1)$ and $\left(\LL^{2 g - 2} - 1\right) \left(\left(\LL - 1\right)^{2 g} - 1\right)/(\LL - 1)$, for any $g \geq 0$ the denominator divides the numerator, so they must be understood formally as the corresponding quotient.
\end{remark}

The theorem above allows us to describe properties of the $\AGL_1(k)$-representation varieties and their character varieties. In the following remarks, we list a few of these results. These results are simple and can be obtained from different approaches as well, however, the results follow naturally from our machinery.

\begin{remark}
    Using Theorem \ref{thm:character_stack_AGL1}, we see that under the evaluation map \eqref{eq:evaluation}
    \[ \ev: \K(\RStck/BG)\to \hat{\K}(\Var_{k}),\]
    the class of the $\AGL_1(k)$-character variety becomes
    \begin{align*}
        \ev([\mathfrak{X}_{\AGL_1(k)}(\Sigma_g)]) &= \frac{1}{q(q-1)}+((q-1)^{2g}-1)\frac{1}{q-1}+\frac{\LL^{2 g} - 1}{\LL - 1}\frac{1}{q}+\frac{\left(\LL^{2 g - 2} - 1\right) \left(\left(\LL - 1\right)^{2 g} - 1\right)}{\LL - 1}\\
        &=\frac{1}{q(q-1)}\left(\LL^{2g}+\LL^{2 g - 1} \left(\left(\LL - 1\right)^{2 g} - 1\right)\right).
    \end{align*}

    Furthermore, since the affine group $\AGL_1(k)$ is a special algebraic group (see Remark \ref{rem:special}), we also have that
    \[ \ev([\mathfrak{X}_{\AGL_1(k)}(\Sigma_g)])=\frac{[R_{\AGL_1(k)}(\Sigma_g, \point)]}{[\AGL_1(k)]}.\]
    In this way, we obtain the class of the representation variety in $\hat{\K}(\Var_{k})$
    \[[R_{\AGL_1(k)}(\Sigma_g)]=\LL^{2g}+\LL^{2 g - 1} \left(\left(\LL - 1\right)^{2 g} - 1\right),\]
    agreeing with the results of \cite{gonlogmun20} and \cite{habvog20}.
    
    Notice that we can compute the same class with a different approach. Consider the morphism $c:\point \to \BG$ given by the trivial torsor. Then, by Lemma \ref{lem:quotient-G-torsor},
    \[c^*[\mathfrak{X}_{\AGL_1(k)}(\Sigma_g, \point)]=[R_{\AGL_1(k)}(\Sigma_g)]\in \K(\RStck_k).\]
    Since 
    \[c^*[BG]=1, \quad c^*[\GG_a/G]=q, \quad c^*[\GG_m/G]=q-1, \quad c^*[\AGL_1(k)/G] = q(q-1)\]
    we have that
    \[[R_{\AGL_1(k)}(\Sigma_g)]=1+q\left((q-1)^{2g}-1\right)+q^{2g}-1+q\left(q^{2g-2}-1\right)\left((q-1)^{2g}-1\right)=\]
    \[=\LL^{2g}+\LL^{2 g - 1} \left(\left(\LL - 1\right)^{2 g} - 1\right).\]
\end{remark}

\begin{remark}
    Theorem \ref{thm:character_stack_AGL1} allows us to give a description of the Luna stratification of the $\AGL_1(k)$-representation variety $R_{\AGL_1(k)}(\Sigma_g)$ with respect to the conjugation action by $\AGL_1(k)$:
    \begin{itemize}
        \item The subvariety on which $\AGL_1(k)$ acts freely is an open, $4g-1$-dimensional subvariety. In particular, the representation variety $R_{\AGL_1(k)}(\Sigma_g)$ has dimension $4g-1$.
        \item There is a subvariety on which $\AGL_1(k)$ acts by scaling, which is a $2g$-dimensional subvariety,
        \item There is a subvariety on which $\AGL_1(k)$ acts by scaling and translation, which is a $(2g + 1)$-dimensional subvariety,
        \item Finally, $\AGL_1(k)$ acts trivially only on a single point.
    \end{itemize}
\end{remark}

\begin{remark}
    Using Theorem \ref{thm:character_stack_AGL1}, we can also describe the $\AGL_1(k)$-character variety, i.e.\ the GIT-quotient $R_{\AGL_1(k)}(\Sigma_g, \point)\sslash \AGL_1(k)$ (see Remark \ref{rmk:character-variety}). In this case, the GIT-quotient can be identified with those points of the representation variety for which the corresponding matrices commute with the subgroup $H$ of diagonal matrices of $\AGL_1(k)$. Therefore, the GIT-quotient can be computed via the functor
    \[(-)^H: \K(\RStck/\BG)\to \K(\RStck_k).\]
    It is easy to see that
    \[ [\BG]^H = 1, \quad [\GG_a/G]^H = 1, \quad [\GG_m/G]^H = [\AGL_1(k)/G]^H = 0 , \] 
    and thus
    \[ [R_{\AGL_1(k)}(\Sigma_g, \point)\sslash \AGL_1(k)] = 1 + (q - 1)^{2g} - 1 = (q - 1)^{2g} .\]
\end{remark}
\newcommand{\genT}{\textup{\bf T}}
\newcommand{\genS}{\textup{\bf S}}

\section[(Gm ⋊ Z/2Z)-character stacks]{$(\GG_m \rtimes \ZZ/2\ZZ)$-character stacks}
\label{sec:GGmZZ2}

In this section, we analyze the geometry of the character stack of the group $G = \GG_m \rtimes \ZZ/2\ZZ$, where the action of $\ZZ/2\ZZ$ on $\GG_m$ is given by $x \mapsto x^{-1}$, over a field $k$ of characteristic $\operatorname{char}(k) \ne 2$. This is a non-connected algebraic group, which leads to some interesting new features of the character stack.

Throughout this section, we will denote the generator of $\ZZ/2\ZZ$ by $\sigma$. In particular, we have $\sigma x \sigma = x^{-1}$ for all $x \in \GG_m$.
Note that $G$ acts on itself by conjugation, and moreover, the normal subgroup $\GG_m \subset G$ and its coset $\GG_m \sigma \subset G$ are both stable under this action. We denote by $[\GG_m / G]$ and $[\GG_m \sigma / G]$ the corresponding quotient stacks, respectively, as well as their classes in $\K(\RStck/\BG)$.



Regarding the relative setting, we shall denote by $\genT \in \K(\RStck/[G/G])$ the class of the inclusion $[\{ 1 \} / G] \subset [G / G]$. Additionally, we shall denote by $\genS \in \K(\RStck/[G/G])$ the class of the morphism $[\GG_m / G] \to [G / G]$ induced by the morphism $\GG_m \to G$ given by $x \mapsto x^2$.


Observe that $\genT = Z\left(\bdpcupleftsmall\right)(1)$. Furthermore, it turns out that the submodule of $\K(\RStck/[G/G])$ generated by $\genT$ and $\genS$ over $\K(\RStck/\BG)$ is invariant under the map $Z'(\bdpgenussmall)$, and their image can be explicitly described.

\begin{proposition}
The $\K(\RStck/\BG)$-submodule $\langle \genT, \genS\rangle \subseteq \K(\RStck/{[G/G]})$ is invariant under the TQFT and the image of the generators is given by
\begin{align*}
    Z'(\bdpgenussmall)(\genT) &= [\GG_m / G]^2 \cdot \genT + 3 [\GG_m \sigma / G] \cdot \genS,\\
    Z'(\bdpgenussmall)(\genS) &= \left([\GG_m / G]+ [\GG_m \sigma / G]\right)^2 \cdot \genS.
\end{align*}
\end{proposition}

\begin{proof}
Let us start with the generator $\genT$. The image $Z'(\genT) = \tilde{q}_! \tilde{p}^* \genT$ is the class of the morphism $[G^2/G] \to [G/G]$ induced by the commutator map $[-, -] \colon G^2 \to G$.
To understand this commutator map, we use the following stratification of the stack $[G^2/G]$,
\[ [G^2/G] = \left(\GG_m \times \GG_m\right) \sqcup  \left(\GG_m \times \GG_m \sigma \right) \sqcup  \left(\GG_m \sigma \times \GG_m \right) \sqcup  \left(\GG_m \sigma \times \GG_m \sigma \right) . \]
In particular, we compute
\[ [x, y] = 1, \quad [x, y \sigma] = x^2, \quad [x \sigma, y] = y^{-2}, \quad [x \sigma, y \sigma] = x^2 y^{-2} , \]
for all $x, y \in \GG_m$. Hence, the first stratum gives a contribution of $[\GG_m / G]^2 \cdot \genT$, and the second and third stratum both give a contribution of $[\GG_m \sigma / G] \cdot \genS$. After a change of variables $\tilde{x} = x^2 y^{-2}$, we see that the fourth stratum also gives a contribution of $[\GG_m \sigma / G] \cdot \genS$.
So, adding up all the contributions, we find
\[ Z'(\bdpgenussmall)(\genT) = [\GG_m / G]^2 \cdot \genT + 3 [\GG_m \sigma / G] \cdot \genS. \]
Next, we focus on the generator $\genS$. The image $Z'(\genT) = \tilde{q}_! \tilde{p}^* \genS$ is the class of the morphism $[\GG_m \times G^2 / G] \to [G/G]$ induced by the map
\[ \GG_m \times G \times G \to G, \quad (z, a, b) \mapsto z^2 [a, b] . \]
Using the same stratification of $[G/G]$ as above, we compute
\[ z^2 [x, y] = z^2, \quad z^2 [x, y \sigma] = x^2 z^2, \quad z^2 [x \sigma, y] = y^{-2} z^2, \quad z^2 [x \sigma, y \sigma] = x^2 y^{-2} z^2 , \]
for all $x, y, z \in \GG_m$.
Hence, the first stratum gives a contribution of $[\GG_m / G]^2 \cdot \genS$. The second and third both give a contribution of $[\GG_m /G] [\GG_m \sigma / G] \cdot \genS$. The fourth stratum gives a contribution of $[\GG_m \sigma / G]^2 \cdot \genS$.
Together, we obtain
\begin{align*}
    Z'(\bdpgenussmall)(\genS)
        &= [\GG_m/G]^2 \cdot \genS + 2 [\GG_m / G] [\GG_m \sigma / G] \cdot \genS + [\GG_m \sigma / G]^2 \cdot \genS , \\
        &= ([\GG_m/G] + [\GG_m \sigma / G])^2 \cdot \genS .
\end{align*}
\end{proof}

The above proposition show that, with respect to the basis $\{ \genT, \genS \}$, we have
\begin{equation}
    \label{eq:GmZZ2_TQFT_1}
    Z'(\bdpgenussmall) = \left[\begin{matrix}
        [\GG_m / G]^2  & 0 \\
        3 [\GG_m \sigma / G] & \left([\GG_m / G]+ [\GG_m \sigma / G]\right)^2
    \end{matrix}\right]
\end{equation}
as a $\K(\RStck/\BG)$-module homomorphism.



In order to apply \eqref{eq:character_stack_from_TQFT}, we must compute powers of the matrix $Z'(\bdpgenussmall)$, and hence we need to describe the product of the classes $[\GG_m / G]$ and $[\GG_m \sigma / G]$ in $\K(\RStck/\BG)$. For the following lemma, we introduce the class $[(\ZZ/2\ZZ) / G] \in \K(\RStck/\BG)$, where $\GG_m$ acts trivially on $\ZZ/2\ZZ$, and $\sigma$ acts transitively.

\begin{lemma}
    In $\K(\RStck/\BG)$, the following relations hold:
    \begin{enumerate}[(i)]
        \item $[\GG_m/G]^2 = (q + 2) [\GG_m/G] - (q - 2) [(\ZZ/2\ZZ)/G] - (q + 1)$
        \item $[\GG_m \sigma/G]^2 = [\GG_m \sigma / G] [\GG_m / G]$
    \end{enumerate}
\end{lemma}
\begin{proof}
    For (i), the action of $G$ on $\GG_m$ can be extended to $\PP^1_k$, so that $[\GG_m/G] = [\PP^1_k/G] - [(\ZZ/2\ZZ)/G]$. After a change of variables on $\PP^1_k$, the action of $G$ can be described by $(x : y) \overset{\sigma}{\mapsto} (-x : y)$. Note that this change of variables uses the assumption that $\operatorname{char}(k) \ne 2$. Hence, $[\PP^1_k/G] = [\AA^1_k/G] + 1$, with $x \overset{\sigma}{\mapsto} -x$ on $\AA^1_k$, and thus $[\GG_m/G] = [\AA^1_k/G] + 1 - [(\ZZ/2\ZZ)/G]$. By simple changes of variables, it is easy to see that $[\AA^1_k/G]^2 = q [\AA^1_k/G]$ and $[\AA^1_k/G] [(\ZZ/2\ZZ)/G] = q [(\ZZ/2\ZZ)/G]$ and $[(\ZZ/2\ZZ)/G]^2 = 2 [(\ZZ/2\ZZ)/G]$. In total, we find
    \begin{align*}
        [\GG_m/G]^2
            &= ([\AA^1_k/G] + 1 - [(\ZZ/2\ZZ)/G])^2 = 1 + [\AA^1_k/G] (q + 2) - 2 q [(\ZZ/2\ZZ)/G] \\
            &= (q + 2) [\GG_m/G] - (q + 1) - (q - 2) [(\ZZ/2\ZZ)/G] .
    \end{align*}
    Now, (ii) follows from the $G$-equivariant isomorphism
    \[ \GG_m \sigma \times \GG_m \sigma \to \GG_m \sigma \times \GG_m, \qquad (x \sigma, y \sigma) \mapsto (x \sigma, \tfrac{y}{x} \sigma) .
    \]
\end{proof}
Using the above lemma, we can express the matrix of \eqref{eq:GmZZ2_TQFT_1} as
\[ Z'\left(\bdpgenussmall\right) = \begin{pmatrix} 
    [\GG_m/G]^2 & 0 \\
    3 [\GG_m \sigma / G] & [\GG_m/G]^2 + 3 [\GG_m \sigma / G] [\GG_m / G]
\end{pmatrix} . \]
More importantly, the above lemma allows us to work in the $\ZZ[q]$-submodule of $\K(\RStck/\BG)$ generated by $[(\ZZ/2\ZZ)/G], [\GG_m / G]$ and $[\GG_m \sigma / G]$. In particular, we can choose generators of $\K(\RStck/[G/G])$ for which we can express $Z'\left(\bdpgenussmall\right)$ as a matrix with coefficients in $\ZZ[q]$. Choosing generators
\[ \left\{ \genT, [(\ZZ/2\ZZ)/G] \cdot \genT, [\GG_m/G] \cdot \genT, [\GG_m \sigma/G] \cdot \genS, [\GG_m \sigma\times \ZZ/2\ZZ / G] \cdot \genS, [\GG_m \sigma / G] [\GG_m / G] \cdot \genS \right\} , \]
the elaborated matrix is given by
\[ Z'\left(\bdpgenussmall\right) = \begin{pmatrix} 
    - q - 1 & 0 & - \left(q + 1\right) \left(q + 2\right) & 0 & 0 & 0 \\
    2 - q & \left(q - 1\right)^{2} & - \left(q - 2\right) \left(2 q + 1\right) & 0 & 0 & 0 \\
    q + 2 & 0 & q^{2} + 3 q + 3 & 0 & 0 & 0 \\
    3 & 0 & 0 & - 4 \left(q + 1\right) & 0 & - 4 \left(q + 1\right) \left(q + 2\right) \\
    0 & 3 & 0 & - 4 \left(q - 2\right) & 4 \left(q - 1\right)^{2} & - 4 \left(q - 2\right) \left(2 q + 1\right) \\
    0 & 0 & 3 & 4 \left(q + 2\right) & 0 & 4 \left(q^{2} + 3 q + 3\right)
\end{pmatrix} . \]
We diagonalize the matrix above. The eigenvalues are given by
\[ 1, \quad 4, \quad (q - 1)^2, \quad (q + 1)^2, \quad 4(q - 1)^2, \quad 4(q + 1)^2 , \]
with respective eigenvectors
\[ \begin{pmatrix}q + 1\\-1\\-1\\- q - 1\\1\\1\end{pmatrix}, \quad\begin{pmatrix}0\\0\\0\\q + 1\\-1\\-1\end{pmatrix}, \quad\begin{pmatrix}0\\\left(q - 1\right)^{2}\\0\\0\\-1\\0\end{pmatrix}, \quad\begin{pmatrix}2 \left(q + 1\right)^{2}\\\left(q - 2\right) \left(q + 1\right)^{2}\\- 2 \left(q + 1\right)^{2}\\-2\\2 - q\\2\end{pmatrix}, \quad\begin{pmatrix}0\\0\\0\\0\\1\\0\end{pmatrix}, \quad\begin{pmatrix}0\\0\\0\\2\\q - 2\\-2\end{pmatrix} . \]

This allows us to compute the virtual class of the character stack.

\begin{theorem}
    \label{thm:character_stack_GGmZZ2}
    For $G = \GG_m \rtimes \ZZ/2\ZZ$, the class of the character stack $\mathfrak{X}_G(\Sigma_g)$ in $\K(\Stck/\BG)$ is given by
    \begin{align*}
        [\mathfrak{X}_G(\Sigma_g)] 
            &= \frac{q + 1 - (q + 1)^{2g}}{q} [\BG] \\
            &\quad + \frac{q \left(q - 1\right)^{2 g} - \left(q - 2\right) \left(q + 1\right)^{2 g} - 2}{2 q} [(\ZZ/2\ZZ)/G] \\
            &\quad + \frac{\left(q + 1\right)^{2 g} - 1}{q} [\GG_m / G] \\ 
            &\quad + \frac{2 \left(4^{g} - 1\right) \left(q - \left(q + 1\right)^{2 g - 2} + 1\right)}{q} [\GG_m \sigma / G] \\
            &\quad + \frac{\left(4^{g} - 1\right) \left(q \left(q - 1\right)^{2 g - 2} - \left(q - 2\right) \left(q + 1\right)^{2 g - 2} - 2\right)}{q} [\GG_m \sigma\times (\ZZ/2\ZZ) / G] \\
            &\quad + \frac{2 \left(4^{g} - 1\right) \left(\left(q + 1\right)^{2 g - 2} - 1\right)}{q} [\GG_m \sigma \times \GG_m / G] .
    \end{align*}
    \qed
\end{theorem}

\begin{remark}\label{rem:repvarggmzz2}
As in Section \ref{sec:AGL1-character-stack}, we can compute the class of the representation variety corresponding to $G=\GG_m\rtimes \ZZ/2\ZZ$ obtaining
    \[ c^*[\mathfrak{X}_{G}(\Sigma_g)] = [R_G(\Sigma_g)] = (q - 1)^{2g - 1} (q - 3 + 2^{2g + 1}) . \]
\end{remark}

We conclude the paper by showing that the naive point counting formula \eqref{eq:pointcstacks} would fail for the group $\GG_m\rtimes \ZZ/2\ZZ$, by computing the image of $[\mathfrak{X}_G(\Sigma_g)]$ under the evaluation map
\[ \ev \colon \K(\RStck/\BG) \to \hat{\K}(\Var_k) . \]
Since this morphism is $\K(\RStck_k)$-linear, it suffices to compute the images of the generators.
\begin{lemma}
    The following identities hold:
    \begin{enumerate}[(i)]
        \item $\ev([\BG]) = q/(q^2 - 1)$,
        \item $\ev([\GG_m / G]) = 1$,
        \item $\ev([\GG_m \sigma / G]) = \ev([\B \{ \pm 1 \}] [\B \langle \sigma \rangle]) = 1$,
        \item $\ev([\ZZ/2\ZZ / G]) = \ev([\B \GG_m]) = 1/(q-1)$,
        \item $\ev([\GG_m \sigma \times \ZZ/2\ZZ / G]) = \ev([\B \{ \pm 1 \}]) = 1$,
        \item $\ev([\GG_m \sigma \times \GG_m / G]) = \ev([ \B \{ \pm 1 \} ] [\GG_m / \langle \sigma \rangle ]) = q$.
    \end{enumerate}
\end{lemma}
\begin{proof}
    \begin{enumerate}[(i)]
        \item We can view $G$ as a subgroup of $\textup{GL}_2$ by identifying $\sigma$ with $\smatrix{0 & 1 \\ 1 & 0}$ and $x \in \GG_m$ with $\smatrix{x & 0 \\ 0 & x^{-1}}$. Then, $\point \times_G \textup{GL}_2 = \textup{GL}_2 \sslash G \to \BG$ is a $\textup{GL}_2$-torsor, so that $[\BG] = [\textup{GL}_2 \sslash G] / [\textup{GL}_2]$. Now,
        \begin{align*}
            \textup{GL}_2 \sslash G
                &= \Spec k[a, b, c, d, (ad - bc)^{-1}]^G \\
                &= \Spec k[ac, bd, ad + bc, (ad - bc)^{-2}] \\
                &\simeq \Spec k[x, y, z, (z^2 - 4xy)^{-1}] ,
        \end{align*}
        whose class in $\K(\Var_k)$ is $q^2(q - 1)$. Hence, $\ev([\BG]) = q^2(q - 1) / [\textup{GL}_2] = q / (q^2 - 1)$.
        \item Similarly, note that $\GG_m \times_G \textup{GL}_2 \to [\GG_m / G]$ is a $\textup{GL}_2$-torsor. We compute the quotient by $G$ in step-wise, first by $\GG_m$ and then by $\ZZ/2\ZZ$:
        \begin{align*}
            \mathbb{G}_m \times_G \textup{GL}_2
                &= \Spec k[a, b, c, d, (ad - bc)^{-1}, x^{\pm 1}]^G \\
                &= \Spec k[ac, ad, bc, bd, (ad - bc)^{-1}, x^{\pm 1}]^{\ZZ/2\ZZ} \\
                &= \Spec k[\alpha = ac, \beta = ad, \gamma = bc, \delta = bd, (\beta - \gamma)^{-1}, x^{\pm 1}]^{\ZZ/2\ZZ} \\
                &= \Spec k[\alpha, \delta, \beta + \gamma, \beta \gamma, (\beta - \gamma)^{-2}, (\beta - \gamma)(x - x^{-1}), x + x^{-1}] \\
                &= \Spec k[\alpha, \delta, w, z, s, t] / (z(w^2 - 4\alpha \delta) - 1, z s^2 - t^2 + 4) ,
        \end{align*}
        whose class is $q(q - 1)^2 (q + 1) = [\textup{GL}_2]$, and hence $\ev([\GG_m/G]) = 1$.
        \item Similar computation can be done for $[\GG_m\sigma/G]$, we leave the details to the reader.
        \item Since $[(\ZZ/2\ZZ)/G] = [\B \GG_m]$ and $\GG_m$ is a special group, we obtain that $\ev([(\ZZ/2\ZZ)/G]) = 1 / (q - 1)$.
        \item Since $[\GG_m \sigma \times \ZZ/2\ZZ / G] = [\B \{ \pm 1 \}]$, we obtain that $\ev([\GG_m \sigma \times \ZZ/2\ZZ / G]) = 1$.
        \item We have that $[\GG_m \sigma \times \GG_m / G] = [\B \{ \pm 1 \}] \cdot [\GG_m / \langle \sigma \rangle]$. The action $t\mapsto t^{-1}$ on $\GG_m$ can be extended to an action on $\PP^1$ by adding the points $0$ and $\infty$. Thus, 
        \[ \ev([\GG_m / \langle \sigma \rangle]) = \ev([\PP^1 / \langle \sigma \rangle]) - 1 = q + 1 - 1 = q \in \hat{\K}(\Var_k).\]
    \end{enumerate}
\end{proof}

Combining this lemma with Theorem \ref{thm:character_stack_GGmZZ2}, we obtain the following corollary.
\begin{corollary}
    \label{cor:GGmZZ2}
    For $G = \GG_m \rtimes \ZZ/2\ZZ$ and any $g \ge 0$, the class of the character stack $\mathfrak{X}_G(\Sigma_g)$ in $\hat{\K}(\Var_k)$ is given by
    \[ \ev([\mathfrak{X}_G(\Sigma_g)]) = \frac{\left(q - 1\right)^{2 g - 2} \left(2^{2 g + 1} + q - 3\right)}{2} + \frac{\left(q + 1\right)^{2 g - 2} \left(2^{2 g + 1} + q - 1\right)}{2} . \]
\end{corollary}

\begin{remark}
    For small values of $g$, we find
    \begin{align*}
        \ev([\mathfrak{X}_G(\Sigma_{0})]) &= \frac{q}{\left(q - 1\right) \left(q + 1\right)}, \\
        \ev([\mathfrak{X}_G(\Sigma_{1})]) &= q + 6, \\
        \ev([\mathfrak{X}_G(\Sigma_{2})]) &= q^{3} + 30 q^{2} + 3 q + 30, \\
        \ev([\mathfrak{X}_G(\Sigma_{3})]) &= q^{5} + 126 q^{4} + 10 q^{3} + 756 q^{2} + 5 q + 126, \\
        \ev([\mathfrak{X}_G(\Sigma_{4})]) &= q^{7} + 510 q^{6} + 21 q^{5} + 7650 q^{4} + 35 q^{3} + 7650 q^{2} + 7 q + 510, \\
        \ev([\mathfrak{X}_G(\Sigma_{5})]) &= q^{9} + 2046 q^{8} + 36 q^{7} + 57288 q^{6} + 126 q^{5} + 143220 q^{4} + 84 q^{3} + 57288 q^{2} + 9 q + 2046.
    \end{align*}
    It is not hard to see that $\ev([\mathfrak{X}_G(\Sigma_{g})])$ is always a polynomial in $q$ for $g\geq 1$. In particular, the $E$-polynomial of the character stack $\mathfrak{X}_G(\Sigma_{g})$ is an integer polynomial in $q=uv$.
\end{remark}

\begin{remark}
    Furthermore, we can compare the class of the representation variety computed in Remark \ref{rem:repvarggmzz2} and the class of the character stack in $\hat{\K}(\Var_k)$ computed in Corollary \ref{cor:GGmZZ2}, and we observe that
    \[ \ev([\mathfrak{X}_G(\Sigma_g)]) \ne \frac{[R_G(\Sigma_g)]}{[G]} \]
    for any $g$, reflecting the fact that $G$ is not connected. This illustrates that one needs to be careful in using naive point counting formula \eqref{eq:pointcstacks} in the case of non-connected groups.
    
    
\end{remark}

\section*{Declarations}

The authors declare that there is no conflict of interest in regards with this paper.

Data sharing is not applicable to this article as no new data were created or analyzed in this study.

\bibliographystyle{abbrv}
\addcontentsline{toc}{section}{References}
\bibliography{bibliography}

\end{document}